\documentclass[11pt]{article}

\usepackage{amsthm}
\usepackage{amsfonts, mathrsfs,amssymb}
\usepackage{amsmath,bbm, soul}
\usepackage[numbers, sort&compress]{natbib}
\usepackage{graphicx}
\usepackage{vmargin, enumerate}
\usepackage{paralist}
\usepackage{times}
\usepackage[pdftex]{hyperref}
\usepackage[latin1]{inputenc}
\usepackage{verbatim}

\newtheorem{thm}{Theorem}[section]

\newtheorem{Def}[thm]{Definition}

\newtheorem{lem}[thm]{Lemma}
\newtheorem{prop}[thm]{Proposition}
\newtheorem{cor}[thm]{Corollary}
\theoremstyle{definition}

\newcommand{\N}{\ensuremath{\mathbb N}}

\newcommand{\R}{\ensuremath{\mathbb{R}}}
\newcommand{\MM}{\ensuremath{\mathbf{M}}}
\newcommand{\FF}{\ensuremath{\mathbf{F}}}
\newcommand{\II}{\ensuremath{\mathbf{I}}}
\newcommand{\XX}{\ensuremath{\mathbf{X}}}
\newcommand{\NN}{\ensuremath{\mathbf{N}}}
\newcommand{\TT}{\ensuremath{\mathbf{T}}}

\newcommand{\RR}{\ensuremath{\mathbf{R}}}
\newcommand{\pS}{\ensuremath{\mathbf{S}}}
\newcommand{\QQ}{\ensuremath{\mathbf{Q}}}

\newcommand{\JJ}{\ensuremath{\mathbf{J}}}

\newcommand{\YY}{\ensuremath{\mathbf{Y}}}
\newcommand{\ZZ}{\ensuremath{\mathbf{Z}}}

\newcommand{\Ms}{\mathcal{M}}
\newcommand{\bo}{\mathrm{O}}
\newcommand{\Var}{\mathrm{Var}}
\newcommand{\Cov}{\mathrm{Cov}}

\newcommand{\sg}{\mathrm{sg}}   

\newcommand{\E}[1]{\ensuremath{\mathbb{E} \left[#1 \right]}}

\newcommand{\Prob}[1]{\ensuremath{\mathbf{P} \left(#1 \right)}}

\newcommand{\I}[1]{\ensuremath{\mathbf{1}_{ \left \{ #1 \right \} }}}

\hypersetup{
    bookmarks=true,         
    unicode=false,          
    pdftoolbar=true,        
    pdfmenubar=true,        
    pdffitwindow=true,      
    pdftitle={Find with increasing median selection},    
    pdfauthor={HS, RN},     
    pdfsubject={Subject},   
    pdfnewwindow=true,      
    pdfkeywords={keywords}, 
    colorlinks=true,       
    linkcolor=black,          
    citecolor=black,        
    filecolor=black,      
    urlcolor=black           
}

\newcommand{\Law}{\mathcal{L}}

\newcommand{\Do} {\mathcal{D}[0,1]}

\newcommand{\Dy} {\mathscr{D}}

\begin{document}

\title{\bf A Gaussian limit process for optimal  FIND algorithms}
\author{Henning Sulzbach\footnote{School of Computer Science,
McGill University,
3480 University Street,
Montreal, Canada H3A 0E9}
\and
Ralph Neininger\footnote{Institute for Mathematics,
J.W.~Goethe University,
60054 Frankfurt am Main,
Germany}  \and
Michael Drmota\footnote{Institute for Discrete Mathematics and Geometry, TU Vienna, A-1040 Vienna, Austria}
}

\maketitle

\begin{abstract}
 We consider versions of the  FIND algorithm  where the pivot element used is the median of a
subset  chosen uniformly at random from the data. For the median selection we assume that subsamples of size asymptotic to  $c \cdot n^\alpha$ are chosen, where
$0<\alpha\le \frac{1}{2}$, $c>0$ and $n$ is the size of the data set to be split.
We consider the complexity of FIND as a process in the rank to be selected and measured by the number of
key comparisons required.  After normalization we show weak convergence of the complexity to a centered Gaussian process as $n\to\infty$, which depends only on $\alpha$. The proof relies on a
contraction argument for probability distributions on c{\`a}dl{\`a}g functions.
We also identify the covariance function of the Gaussian limit process and discuss path and tail properties.
\end{abstract}


\section{Introduction}
The FIND algorithm is a selection algorithm, also called Quickselect, to find an element of given rank $\ell$ in a set $S$ of data, where the data set $S$  is a subset of finite cardinality $|S|$ of some ordered set. We have $\ell\in\{1,2,\ldots,|S|\}$ and assume that the data are  distinct. The algorithm was  introduced by Hoare \cite{ho61}.

FIND is a one-sided version of the well-known sorting algorithm Quicksort. It works recursively by first choosing one element $p\in S$, called the pivot element, and generating two subsets $S_<$ and $S_>$, where $S_<:=\{s\in S\,|\, s<p\}$ and $S_>:=\{s\in S\,|\, s>p\}$.
If $\ell=|S_<|+1$ then the pivot element is the rank $\ell$ element to be selected and the algorithm stops. Otherwise, if $\ell\le |S_<|$ it is recursively applied to $S_<$, if $\ell\ge |S_<|+2$ it is recursively applied the $S_>$ searching for rank $\ell-|S_<|-1$. This is called the $3$-version of the algorithm, since the first partitioning step leads to three cases. A variant is the $2$-version, where in the first partitioning step the sets $S_\le:=\{s\in S\,|\, s\le p\} $ and $S_>$ are generated. Note that we have $p\in S_\le$. We ignore the case where the pivot element is the rank $\ell$ element and recursively apply the algorithm to the subset among $S_\le$ and  $S_>$ where the rank $\ell$ element is contained.
Actually we will discuss both versions of the algorithm.

This specifies the FIND algorithm except for the choice of the pivot element in the partitioning step.
It can be chosen as the first element of $S$,
if $S$ is given as a list (vector) so that a first
element is well-defined, it can as well be chosen
uniformly at random from $S$. In order to obtain
better balanced subsets $S_<$ and $S_>$, respectively
$S_\le$ and $S_>$, one may first choose a subset $M$ of
odd cardinality $k$ from $S$ and use the median of
$M$ as pivot element. This version is called the
median-of-$k$ FIND algorithm. Here $k$ is fixed in
advance and constant until the algorithm has
performed all recursive calls and stops. A
variant of FIND, which is discussed  in the
present paper, consists of letting $k=k(n)$ depend
on $n$ so that $1\le k(n)\le n$ is odd and grows
asymptotically as $c \cdot n^\alpha$ where $c>0$,
$0<\alpha\le \frac{1}{2}$ and $n\to\infty$.  Note
that in a recursive call on some $S'\subset S$  the
subset of $S'$ to choose the median from is of the
 size $k(|S'|)$.  These routines turn out to be asymptotically optimal in a sense described below. First discussions of such
 versions can be found in \cite{gr99} and more systematically in \cite{maro01}.

The algorithmic  motivation
for this  version is to obtain even more
balanced sublists. This results in algorithms which
are efficient uniformly over the rank $\ell$, hence
they are reliable as universal algorithms to search
for any rank $1\le \ell\le n$. Note that one could
also adapt the algorithm to select particular ranks
$\ell$.    This is a different task; the literature
is reviewed below.

For our subsequent probabilistic analysis we assume that the data are random variables in the unit interval $[0,1]$, which are independent and identically distributed all with the uniform distribution on $[0,1]$. Note that all our results also hold for any deterministic set of data as long as the subset to select the pivot element in each step is chosen independently and uniformly from the set of data. In our probabilistic model we also assume that the subset for the pivot selection is chosen  independently from the data.

As a measure for the complexity we consider the number of key comparisons required by the  version of FIND.
We denote by $X_n^{(2)}(\ell)$ and $X_n^{(3)}(\ell)$ the number of key comparisons required when starting with a set of size $n$ and selecting the element with rank $1\le  \ell \le n$ using the $2$-version and $3$-version respectively. Note that the choice of $c$ and $\alpha$ as well as the particular choice of the median selection algorithm to find the pivot element within the subset
are suppressed in the notation. A median of a set can be found in time (i.e.~number of key comparisons) linear in the size of the set. It will later
 turn out that our results are independent of the choice of the median-selection algorithm to find the pivot element within the random subset as long
as mild assumptions are satisfied which are shared by standard median-selection algorithms (we could in fact use FIND itself in this step). We denote
the number of key comparisons needed to find the pivot as the median of a subset  of size $k=k(n)$ by $T_n$ and assume   for any $p \geq 1$ that
we have
\begin{align}\label{rn_con_med}
\|T_n\|_p=\bo(k(n)), \quad (n\to \infty),
\end{align}
where $\|X\|_p:=\E{|X|^p}^{1/p}$ denotes the $L_p$-norm of a random variable $X$ for $1\le p<\infty$. The  big-$\bo$ notation as well as other Bachmann--Landau symbols are used here and later on.

The rank parameter $\ell$ is subsequently also interpreted as a time parameter of a stochastic process and we denote $X^{(2)}_n:=(X^{(2)}_n(\ell))_{1\le \ell \le n}$ and $X^{(3)}_n:=(X^{(3)}_n(\ell))_{1\le \ell \le n}$.
 In Theorems \ref{rn_thm_in1} and \ref{rn_thm_in2} we state our main results about the asymptotic behavior of $X^{(2)}_n$ and $X^{(3)}_n$. Subsequently, we consider all appearing stochastic processes in time $t\in[0,1]$ with c{\`a}dl{\`a}g paths as random elements of the space $(\Do,d_{sk})$
 of c{\`a}dl{\`a}g functions on $[0,1]$ with the Skorokhod metric $d_{sk}$, see Billingsley \cite[Chapter 3]{Billingsley1999}.

\begin{thm} \label{rn_thm_in1}
Consider the process $X_n^{(2)}= (X_n^{(2)}(\ell))_{1\le \ell\le n}$ of the number of key comparisons needed by the $2$-version of the median-of-$k$ FIND algorithm with $k=k(n)\sim c n^\alpha$ with $c>0$ and $\alpha\in(0,\frac{1}{2}]$ and condition (\ref{rn_con_med}) for the pivot selection in the partitioning step. Then we have, as $n\to \infty$,  the weak convergence
 \begin{align*}
  \left(\frac{X_n^{(2)}(\lfloor tn \rfloor + 1)  -2n}{ n^{1-\alpha/2}/\sqrt{c}} \right)_{t\in[0,1]} \stackrel{d}{\longrightarrow} Z \quad \mbox{ in } \; (\Do,d_{sk}),
\end{align*}
where $Z=(Z_t)_{t\in[0,1]}$ is a centered Gaussian process depending on $\alpha$ with covariance function specified in Theorem \ref{rn_gau_pro} below (and where we set
by convention $X_n^{(2)}(n+1):=X_n^{(2)}(n)$).
\end{thm}
Our main convergence result for the $3$-version is the weak convergence of all finite dimensional marginals, denoted by $\stackrel{fdd}{\longrightarrow}$, for the analogously normalized  process  to the corresponding marginals of the Gaussian process of Theorem \ref{rn_thm_in1}.
\begin{thm} \label{rn_thm_in2}
Consider the process $X_n^{(3)}= (X_n^{(3)}(\ell))_{1\le \ell\le n}$ of the number of key comparisons needed by the $3$-version of the median-of-$k$ FIND algorithm with $k=k(n)\sim c n^\alpha$ with $c>0$ and $\alpha\in(0,\frac{1}{2}]$ and condition (\ref{rn_con_med}) for the pivot selection in the partitioning step. Then we have, as $n\to \infty$,   convergence of the finite dimensional marginals,
 \begin{align*}
  \left(\frac{X_n^{(3)}(\lfloor tn \rfloor + 1)  -2n}{ n^{1- \alpha/2}/\sqrt{c}} \right)_{t\in[0,1]} \stackrel{fdd}{\longrightarrow} Z,
\end{align*}
where $Z=(Z_t)_{t\in[0,1]}$ is the centered Gaussian process of Theorem \ref{rn_thm_in1}
 (and where we set
by convention $X_n^{(3)}(n+1):=X_n^{(3)}(n)$).
\end{thm}
Some additional related results are stated in Corollary \ref{cormain2}.

As observed by Gr\"ubel \cite{gr99}, for the worst-case behavior of any version of FIND, we have $$ \liminf_{n \to \infty}
\sup_{1 \leq \ell \leq n} \frac{X_n^{(3)}(\ell)}{n} \geq 2.$$ Moreover,
Gr\"ubel \cite[Theorem 5]{gr99} notes that $\frac 1 n \sup_{1 \leq \ell \leq n} X_n^{(3)}(\ell) \to 2$ in probability for any median-of-$k$
FIND variant with $k = k(n) \to \infty$ and $k = o(n/\log n)$. Hence, the algorithms investigated in the present work are asymptotically optimal with respect to the
worst-case behavior.  The following theorem gives more precise information.
\begin{thm} \label{thm_sup}
 As $n \to \infty$, with convergence of all moments, we have
$$\frac{\sup_{1 \leq \ell \leq n} X^{(3)}_n(\ell) - 2n }{n^{1-\alpha/2}/\sqrt{c}} \to \sup_{t \in [0,1]} Z(t),$$
where $Z(t)$ is the process of Theorem \ref{rn_thm_in1}.
The same result holds for the $2$-version.
\end{thm}

In the classical case of FIND (by classical we mean with a uniformly chosen pivot element) a process convergence results for the number of
key comparisons (as in Theorems \ref{rn_thm_in1} and \ref{rn_thm_in2})
 has been obtained in the seminal paper of Gr\"ubel and R\"osler \cite{grro96}.
More precisely, if $X_n(\ell)$ denotes the number of key comparisons
(in the $2$-version) in classical FIND then
 \begin{align}\label{findgrro}
  \left(\frac{X_n(\lfloor tn \rfloor + 1)}{n} \right)_{t\in[0,1]} \stackrel{d}{\longrightarrow}
(\widetilde{Z}(t))_{t\in[0,1]} \quad \mbox{ in } \; (\Do,d_{sk}),
\end{align}
where $\widetilde Z:=(\widetilde{Z}(t))_{t\in[0,1]}$ satisfies the stochastic fixed point equation
\begin{align*}
(\widetilde{Z}(t))_{t\in[0,1]}\stackrel{d}{=} \left( 1 + {\bf 1}_{[0,U)}(t)\, U \widetilde Z_0\left( \frac{t}U \right) +
 {\bf 1}_{[U,1]}(t)\, (1-U) \widetilde Z_1\left( \frac{t-U}{1-U} \right)\right)_{t\in[0,1]}.
\end{align*}
Here, $\widetilde Z_0$ and $\widetilde Z_1$ have the same distribution as  $\widetilde Z$,
$U$ is uniformly distributed on $[0,1]$, and $\widetilde Z_0, \widetilde Z_1,U$
are independent. In \cite{grro96} also the difference between the $2$-version and
$3$-version is discussed regarding weak convergence in $(\Do,d_{sk})$ for the $2$-version, whereas for the $3$-version such a convergence does not hold. A similar behavior appears for our  FIND algorithm as reflected in Theorems \ref{rn_thm_in1} and \ref{rn_thm_in2}.

For the classical FIND  Paulsen \cite{pa97} studied variances and higher moments  in the setting of quantiles of \cite{grro96}.   Kodaj and M\'ori \cite{komo97} investigated rates of convergence for the marginals of the process.  Hwang and Tsai \cite{hwts02} considered the case $t=0$, i.e.~ranks of the form $\ell=o(n)$ and found (among other things) that here the limit distribution is the Dickman distribution. Note that this is the distribution of $\widetilde{Z}(0)$.

With respect to the one-dimensional marginals, Theorem \ref{rn_thm_in1} and Theorem \ref{rn_thm_in2} reveal that, asymptotically, both first and second order behavior of the considered complexities do not depend on $t \in [0,1]$. This stands in sharp contrast to the results for classical FIND (and median-of-$k$ FIND with $k > 1$ fixed reviewed below), as the distribution of $\widetilde{Z}(t)$ in \eqref{findgrro} depends on $t$.

Historically the mathematical analysis of classical FIND was initiated with an average case analysis for fixed ranks $\ell$ by Knuth \cite{kn72}. Variances were derived in Kirschenhofer and Prodinger \cite{kipr98}.

For mathematical analysis of median-of-$k$ versions of FIND with fixed $k$ not depending on the size of the input we refer to
Anderson and Brown \cite{anbr92}, Kirschenhofer, Mart{\'{\i}}nez and Prodinger \cite{kiprma97} and Gr\"ubel \cite{gr99}. A broad survey, also covering median-of-$k$ analysis is given in R\"osler \cite{ro04}.

A discussion of FIND versions with $k=k(n)$ depending on the size $n$ of the list to be split with respect to the worst-case behavior was given in Gr\"ubel \cite{gr99}. Mart{\'{\i}}nez and Roura \cite{maro01} give an average case analysis,
where optimal choices for the tradeoff between better balanced sublists versus additional
cost for the median selection are discussed. Note that another idea to adapt the FIND algorithm is to not
choose the median of a subsample but to choose an element that may depend on the rank $\ell$ searched for such that the sublist where the algorithm is recursively called may be small. This is investigated in Mart{\'{\i}}nez,  Panario and Viola \cite{mapavi10}, see also Knof and R\"osler \cite[pp. 151--153]{knro12}.

In various contributions also the number of key exchanges is studied which has to be compared with the number of key comparisons for a more realistic measure of complexity.
Corresponding limit distributions can be found in
Hwang and Tsai \cite{hwts02}, Knape and Neininger \cite[Section 5]{knne08}, Mahmoud \cite{ma09,ma10} and Dadoun and Neininger \cite{dane14}.

Another model for the rank searched for is to consider a random rank chosen uniformly and independently from the data and algorithm. So called grand averages where considered for key comparisons in
 Mahmoud, Modarres and Smythe \cite{momosm95},
  and, for a different version of the partitioning stage using two pivot elements, in Wild, Nebel and Mahmoud \cite{winema13}. For the number of key exchanges under grand averages see \cite{ma10, dane14}. Yet another complexity measure is the worst case complexity with worst case over the possible ranks, see Devroye \cite{de01b}.

Tail bounds for the number of key comparisons for the classical FIND were studied in  Devroye \cite{de84} and Gr\"ubel \cite{gr98}.

A fundamentally different cost measure arises when a key comparison is weighted by the number of bit comparisons needed to identify its result. The number of  bit comparisons was studied by Vall\'ee et al. \cite{vaclfifl09} and Fill and Nakama \cite{fina10,fina12}, see also Grabner and Prodinger \cite{grpr08}.

Finally we mention studies of exact simulation from distributions appearing as limit distributions in the analysis of FIND: Devroye \cite{de01a}, Fill and Huber \cite{fihu10}, Devroye and Fawzi \cite{defa10}, Devroye and James \cite{deja11} and \cite{blsi11,knne12,dane14}.\\

The techniques used to show convergence in Section \ref{rn_sec3} and to construct the limit process $Z$ in Section \ref{rn_sec21} are in the spirit of the contraction method.
(We refer to R\"osler and R\"uschendorf \cite{RoRu01} and Neininger and R\"uschendorf \cite{NeRu04} for an introduction and survey of the contraction method for univariate and finite-dimensional quantities.) In the last years a couple of general approaches have  been developed to show process convergences within the contraction method on different function spaces and in different topologies, see Eickmeyer and R\"uschendorf \cite{EiRu07}, Drmota, Janson and Neininger \cite{DrJaNe08}, Knof and R\"osler \cite{knro12}, Neininger and Sulzbach \cite{nesu12} and Ragab and R\"osler \cite{raro13}, as well as  the PhD theses of Knof \cite{kn07}, Ragab \cite{ra11} and Sulzbach \cite{SuDiss}.

The construction of the limit process $Z$ that we present in Section \ref{rn_sec21}
builds upon ideas of  Ragab and R\"osler \cite{raro13}. However, the convergence proof for
Theorem \ref{rn_thm_in1} yields weak convergence in $(\Do,d_{sk})$ which has to be compared with
the convergence of finite dimensional distributions shown for a related problem in \cite{raro13}.
Our approach to convergence is almost entirely based on contraction arguments on the level of the
supremum norm of processes and very little (deformation of time) is needed in addition  to align
jumps. Besides 
leading to comparatively strong results, we feel that the technique for
convergence  developed here   is   flexible and general  to be easily applicable to related
recursive problems.

A similar version of the Quicksort algorithm consists in also choosing the pivot element in each step  as a median of a random sub-sample of size $k = k(n) \sim c n^\alpha$ with $n$ the size of the list to be split. We conjecture that such a Quicksort algorithm admits a Gaussian limiting distribution for the normalized number of key comparisons. This would be in contrast to the well-known non-Gaussian limiting distribution
for classical Quicksort, see \cite{ro91}.

\bigskip

{\bf Plan of the paper.}
The paper is organized as follows. In Section \ref{rn_sec21} the limit process $Z$ is constructed and in Section \ref{rn_sec22} identified as a centered Gaussian process with explicitly given covariance function. Section \ref{rn_sec3} contains the asymptotic analysis of the complexity of the  median-of-$k$ FIND leading to the proofs of Theorems \ref{rn_thm_in1} and \ref{rn_thm_in2}. The organization of the proofs is outlined at the beginning of Section \ref{rn_sec3}.
In the final Section \ref{rn_sec4} we present properties of the limit process $Z$. In Subsections \ref{rn_sec41} and \ref{rn_sec42} path properties of $Z$ are discussed, Subsection \ref{rn_sec43} has a characterization and a tail bound  for the supremum of the limit process $Z$.  The Appendix is devoted to the proofs of two technical lemmata. The first, Lemma \ref{lem3}, allows the transfer of the results for the 2-version in Theorem \ref{rn_thm_in1} to the 3-version in Theorem \ref{rn_thm_in2}. The second, Lemma \ref{lem:helppq}, is needed in the study of the path variation of the limit process $Z$.

\bigskip
{\bf Acknowledgements:} We thank the referees for their careful reading and constructive remarks.

\section{Construction and characterization of the limit process}
We first construct and characterize the limit process $Z$ appearing in Theorems \ref{rn_thm_in1} and \ref{rn_thm_in2}. In this and the following section
we fix $\alpha \in (0,1/2]$ and suppress the dependence on $\alpha$ in the notation.
\subsection{Construction} \label{rn_sec21}
We consider the rooted complete infinite binary tree, where the root is labeled by the empty word $\epsilon$ and left and right children of a node labeled $\vartheta$ are labeled by the extended words $\vartheta 0$ and $\vartheta 1$ respectively. The set of labels is denoted by $\Theta:=\cup_{k=0}^\infty \{0,1\}^k$. The length $|\vartheta|$ of a label of a node is identical to the depth of the node in the rooted complete infinite binary tree.

We denote the supremum norm on $\Do$ by $\|\,\cdot\,\|$. For a random variable $X$ in $(\Do, d_{sk})$ and $1\le p<\infty$ we denote the $L_p$-norm by $\|X\|_p:=(\E{\|X\|^p})^{1/p}$.

For $u\in[0,1]$ we define linear operators
\begin{align*}
\mathfrak{A}_u,\mathfrak{B}_u: \Do \to \Do
\end{align*}
as follows. For $f\in \Do$ the c{\`a}dl{\`a}g functions $\mathfrak{A}_u(f)$ and $\mathfrak{B}_u(f)$ are defined as
\begin{align*}
t\mapsto \I{t <u} f \left(\frac t u \right),  \quad t\mapsto \I{t \geq u} f \left(\frac{1-t}{1-u}\right),
\end{align*}
respectively. Furthermore, we define the step function $\sg\!\!:[0,1]\to \R$ by $\sg(t)=\I{t < 1/2} - \I{t  \geq 1/2}$. Hence, $\sg$ is a shifted version the sign function, and it is in $\Do$.

For a given family $\{N_\vartheta\,|\, \vartheta\in\Theta\}$ of independent random variables in $\R$ each with the standard normal distribution
we recursively define a family $\{Z^\vartheta_n\,|\, \vartheta\in\Theta, n\in\N_0\}$ of random variables in $(\Do,d_{sk})$ as follows: We set $Z_0^\vartheta:= 0$ for all $\vartheta\in\Theta$. Assume, the $Z_n^\vartheta$ are already defined for an $n\ge 0$ and all $\vartheta\in\Theta$. Then for all  $\vartheta\in\Theta$ we set
\begin{align}\label{recdefproc}
Z_{n+1}^\vartheta := \left( \frac 1 2 \right)^{1-\alpha/2}
\mathfrak{A}_\frac{1}{2}( Z_n^{\vartheta 0})
+  \left( \frac 1 2 \right)^{1-\alpha/2} \mathfrak{B}_\frac{1}{2}(Z_n^{\vartheta 1}) +  N_\vartheta \cdot\sg.
\end{align}
We have the following asymptotic properties for the $Z_n^{\vartheta}$:
\begin{lem} \label{rnlem}
 Let $\{Z^\vartheta_n\,|\, \vartheta\in\Theta, n\in\N_0\}$ be a family as defined (\ref{recdefproc}).
Then, for each $\vartheta \in \Theta$, the sequence $(Z^\vartheta_n)_{n\ge 0}$ converges almost surely uniformly and in the $L_p$-norm for all $p \in \N$ to a limit c{\`a}dl{\`a}g process $Z^\vartheta$.
For all $\vartheta\in\Theta$ we have, almost surely,
\begin{align}\label{rn2}
Z^\vartheta = \left( \frac 1 2 \right)^{1-\alpha/2}
\mathfrak{A}_\frac{1}{2}( Z^{\vartheta 0}) +
\left( \frac 1 2 \right)^{1-\alpha/2} \mathfrak{B}_\frac{1}{2}(Z^{\vartheta 1}) +  N_\vartheta\cdot \sg.
\end{align}
The family $\{Z^\vartheta\,|\, \vartheta\in\Theta\}$ is identically distributed and all moments of the $\|Z^\vartheta\|$ are finite.
 \end{lem}
 \begin{proof}
We first show by induction that for all $\vartheta\in\Theta$ and all $n\in\N_0$ we have
\begin{align}\label{rn1}
\E{\|Z_{n+1}^\vartheta - Z_n^\vartheta\|^2} \leq 2^{-(1-\alpha)n}.
\end{align}
For $n=0$ and $\vartheta\in\Theta$ we have
$\E{\|Z_{1}^\vartheta - Z_0^\vartheta\|^2}=\E{|N_\vartheta|^2}=1$, so (\ref{rn1}) is satisfied for $n=0$. Now as induction hypothesis assume, that (\ref{rn1}) is true for all $\vartheta\in\Theta$ with  $n$ replaced by $n-1$.
Note that for a random variable $X$ in $\Do$ we have $\E{\|X\|^2}=\E{\|X^2\|}$ and that for all $f,g\in \Do$ we have $\mathfrak{A}_u(f)\mathfrak{B}_u(g)=0$ and $\|\mathfrak{A}_u(f)\|=\|\mathfrak{B}_u(f)\|=\|f\|$. With these properties, (\ref{recdefproc}) and the induction hypothesis we obtain
\begin{align}\label{rn3}
\E{\|Z_{n+1}^\vartheta - Z_n^\vartheta\|^2}
& \leq\E{\left\| \left(\frac 1 2 \right)^{1-\alpha/2}
\mathfrak{A}_\frac{1}{2}( Z^{\vartheta 0}_n - Z^{\vartheta 0}_{n-1})\right\|^2 +  \left\| \left(\frac 1 2 \right)^{1-\alpha/2}
\mathfrak{B}_\frac{1}{2}( Z^{\vartheta 1}_n
- Z^{\vartheta 1}_{n-1})\right\|^2} \nonumber\\
&=\left(\frac 1 2 \right)^{2-\alpha}\left\{\E{\|  Z^{\vartheta 0}_n - Z^{\vartheta 0}_{n-1}\|^2} + \E{\|  Z^{\vartheta 1}_n - Z^{\vartheta 1}_{n-1}\|^2}  \right\}\nonumber\\
&\leq \left(\frac 1 2 \right)^{2-\alpha} 2 \cdot 2^{-(1-\alpha)(n-1)}=2^{-(1-\alpha)n}.
\end{align}
From \eqref{rn1}, using Markov's inequality, we infer that $\sup_{m \geq n} \| Z^\vartheta_m - Z^\vartheta_n\| \to 0$ as $n \to \infty$ in probability and hence $\sup_{m,p \geq n} \| Z^\vartheta_m - Z^\vartheta_p\| \to 0$ as $n \to \infty$ in probability by a simple application of the triangle inequality. By monotonicity, the latter convergence is almost sure. In other words, for each $\vartheta\in\Theta$, the sequence $(Z^\vartheta_n)_{n\ge 0}$ is almost surely a Cauchy sequence with
respect to the $\|\,\cdot\,\|$-norm.
Since $(\Do, \| \cdot \|)$ is complete, there is a
limit random process $Z^\vartheta$ such that we have
convergence almost surely uniformly.

Since the
operators $\mathfrak{A}_\frac{1}{2}$ and $\mathfrak{B}_\frac{1}{2}$ are
continuous with respect to the $\|\,\cdot\,\|$-norm we
obtain (\ref{rn2}) from (\ref{recdefproc}) by letting
$n\to\infty$. By construction,
$\{Z_n^\vartheta\,|\, \vartheta \in \Theta\}$ is a family of identically distributed random variables for
each $n\in\N_0$. Hence we obtain that the
$Z^\vartheta$ are identically distributed.
  Finally, note that $Z_n^\vartheta = Z_0^\vartheta + \sum_{k=1}^n Z_k^\vartheta - Z_{k-1}^\vartheta$. Using \eqref{rn1} and the triangle inequality for the $\|\, \cdot \, \|_2$-norm implies that
$\E{\|Z_n^\vartheta\|^2}$ is bounded. The same arguments applied to the decomposition $Z^\vartheta  = Z_0^\vartheta + \sum_{k=1}^\infty Z_k^\vartheta - Z_{k-1}^\vartheta$ show that
$\E{\|Z_n^\vartheta - Z^\vartheta\|^2} \to 0$. 
Similar arguments apply for higher  moments. 
 \end{proof}
\begin{Def}\label{rn_nota}
We write $Z:=Z^\epsilon$, hence $Z$ is a random process identically distributed as the $Z^\vartheta$ in Lemma \ref{rnlem} and call it the limit process and its distribution the limit distribution. Analogously we define $Z_n:=Z^\epsilon_n$.
\end{Def}
Let $\mathcal{M}$ denote the set of probability measures on $(\Do,d_{sk})$. We define the map $T : \mathcal{M} \to \mathcal{M}$ by, for $\mu \in \mathcal{M}$,
\begin{align}\label{rn_mapT}
 T(\mu) := {\cal L} \left( \left( \frac{1}{2} \right)^{1-\alpha/2} \mathfrak{A}_{\frac{1}{2}}(X_0) +  \left( \frac{1}{2} \right)^{1-\alpha/2} \mathfrak{B}_{\frac{1}{2}}(X_1) + N\cdot \sg \right),
\end{align}
where $\Law(X_0) = \Law(X_1) = \mu$, $N$ has the standard normal distribution and  $X_0,X_1, N$ are independent.  For $1\le p <\infty $,
 we further denote
\begin{align*}
\mathcal{M}_p(\Do):=\Big\{\mu\in\mathcal M(\Do) \,\Big|\, \int \|x \|^p d\mu(x) <  \infty\Big\}.
\end{align*}
Let
\begin{align} \label{const_p} p_\alpha = \frac{2}{2-\alpha}.
\end{align}
We have the following characterization of the limit distribution $\Law(Z)$ of $Z$:
\begin{lem}
Let $p > p_\alpha$. The limit distribution $\Law(Z)$ from Definition \ref{rn_nota} is the unique fixed-point of the restriction
of $T$ to $\mathcal{M}_p(\Do)$.
\end{lem}
\begin{proof}
It is clear that $T(\mathcal{M}_p(\Do)) \subseteq \mathcal{M}_p(\Do)$.
We endow $\mathcal{M}_p(\Do)$ with the following metric $d$: For $\mu, \nu \in \mathcal{M}_p(\Do)$ let
\begin{align*}
 d(\mu, \nu) := \inf\left\{ \left( \E { \|X-Y\|^p}\right)^{1/p} : \Law(X)= \mu, \Law(Y) =\nu\right\}.
\end{align*}
To see that the restriction of $T$ to $\mathcal{M}_p(\Do)$ is a strict contraction with respect to $d$ let $\mu, \nu \in \mathcal{M}_p(\Do)$ be arbitrary,
fix $\varepsilon > 0$ and choose random processes $X$ and $Y$ with $\Law(X) = \mu$, $\Law(Y) = \nu$ and
$\left(\E{ \|X-Y\|^p}\right)^{1/p} \leq d(\mu, \nu) + \varepsilon$. Let $(X',Y')$ be a copy of $(X, Y)$ such that $N, (X, Y), (X',Y')$ are independent
and $N$ has the standard normal distribution. Then a calculation similar to (\ref{rn3}) implies
\begin{align*}
 d^p(T(\mu), T(\nu)) & \leq  \left( \frac{1}{2} \right)^{p(1-\alpha/2)}  \left(  \E{\|X - Y\|^p} + \E {\|X' - Y'\|^p}\right) \\
 &\leq 2^{1-p(1-\alpha/2)} (d(\mu, \nu) + \varepsilon)^p.
\end{align*}
With $\varepsilon \downarrow 0$ we obtain $d(T(\mu), T(\nu)) \leq 2^{1/p-(1-\alpha/2)} \: d(\mu, \nu)$. Hence, the restriction of $T$ to
$\mathcal{M}_p(\Do)$ is a strict contraction and has at most one fixed point. This implies the assertion.
\end{proof}

\subsection{Characterization of the limit process}\label{rn_sec22}
For $\vartheta \in \Theta$ let $B_\vartheta$ be the set of real numbers in $[0,1]$ whose binary representation has prefix $\vartheta$. Here, the binary expansion of $t= t_1 t_2 \ldots  \in [0,1)$ is uniquely determined by the convention that we always use expansions such that for all $k \in \N$ there exists $\ell > k$ with $t_\ell= 0$. Note that we have the decomposition $B_\vartheta =B_{\vartheta 0} \cup B_{\vartheta 1}$.
The construction in (\ref{recdefproc}) with the $N_\vartheta$ there implies  representations for $Z$ and $Z_n$ from Definition \ref{rn_nota}, for all $t\in[0,1]$ and $n\ge 0$:
\begin{align}
& Z_n(t) = \sum_{\vartheta \in \Theta:\; |\vartheta| < n} \left( \frac 1 2 \right)^{(1-\alpha/2) \cdot |\vartheta|} \left( \I{t \in B_{\vartheta 0}} - \I{t \in B_{\vartheta 1}}\right) N_\vartheta \nonumber\\
& Z(t) = \sum_{\vartheta \in \Theta} \left( \frac 1 2 \right)^{(1-\alpha/2) \cdot |\vartheta|} \left( \I{t \in B_{\vartheta 0}} - \I{t \in B_{\vartheta 1}}\right) N_\vartheta.\label{rn_z_rep}
\end{align}
Thus, $Z_n$ is constant on  the intervals $[i2^{-n}, (i+1)2^{-n})$ for $i =0, \ldots, 2^n -1$. The $\vartheta \in\Theta$ with $|\vartheta|=n$ we denote in lexicographical order by $w_0, w_1, \ldots, w_{2^n-1}$. Then we have
\begin{align*} Z_{n+1}(t) - Z_n(t) & = \left( \frac 1 2 \right)^{(1-\alpha/2) \cdot n}\sum_{i = 0}^{2^n-1} \left( \I{t \in B_{w_i 0}} - \I{t \in B_{w_i 1}}\right) N_{w_i} \nonumber\\
& = \left( \frac 1 2 \right)^{(1-\alpha/2) \cdot n}\sum_{j = 0}^{2^{n+1}-1} \I{j2^{-(n+1)} \leq t < (j+1) 2^{-(n+1)}} (-1)^j N_{w_{\lfloor j/2 \rfloor}}.
 \end{align*}
For $u,v \in [0,1]$  we denote their binary expansions by
\begin{align*}
 u = \sum_{i=1}^{\infty} u_i2^{-i}, \quad v = \sum_{i=1}^{\infty} v_i2^{-i},
\end{align*}
with $u_i, v_i \in \{0,1\}$, again with the convention  introduced above. Then we denote the length of the longest joint prefix of $u$ and $v$ in their binary expansions  by
\begin{align*}
 j(u,v) = \max\{j \geq 1: (u_1, \ldots, u_j) = (v_1, \ldots, v_j)\},
\end{align*}
with the conventions $\max \emptyset := 0$ and $\max \N :=\infty$.
\begin{thm}\label{rn_gau_pro}
The limit process $Z$ from Definition \ref{rn_nota} is  a centered Gaussian process with c{\`a}dl{\`a}g paths. For its covariance function $\sigma(s,t):=\E {Z(t) Z(s)}$ we have
\begin{align}\label{rn5}
 \sigma(s,t) = \frac{\kappa^{j(s,t)+1}-2 \kappa^{j(s,t)} +1}{1-\kappa}, \quad \kappa := \left( \frac 1 2 \right)^{2-\alpha}
\end{align}
with the convention $\kappa^{\infty} := 0$. Equivalently,
\begin{align}\label{rn4}
 \E {(Z(t)- Z(s))^2} = \gamma  \kappa^{j(s,t)}, \quad \gamma = \frac{4 - 2\kappa }{1-\kappa}.
\end{align}

\end{thm}
\begin{proof}
By induction we find that  $(Z_n)_{n\ge 0}$ is a sequence of centered Gaussian processes. Hence, Lemma \ref{rnlem} implies that $Z$ is a centered Gaussian process.  It remains to compute the covariance function of $Z$. Comparing left and right hand side of equation (\ref{rn2}) and using that, by construction, $N_\vartheta$, $Z^{\vartheta 0}$ and $Z^{\vartheta 1}$ are independent, we find
\begin{align*}
\sigma(s,t) = \left\{ \begin{array}{cl} \kappa \, \sigma(2s,2t) +1, &\mbox{if } 0 \leq s,t < 1/2,\\
\kappa \,\sigma(2s-1,2t-1) +1, &\mbox{if }  1/2 \leq s,t \leq 1, \\ -1, &\mbox{if }  0 \leq s < 1/2 \leq t \leq 1. \end{array} \right.
\end{align*}
From this it follows, for $s\neq t$ that
\begin{align*}
 \sigma(s,t) = -\kappa^{j(s,t)} + \sum_{i=0}^{j(s,t) - 1} \kappa^i.
\end{align*}
By the theorem of dominated convergence, right-continuity of $Z$ and the fact that $\E{\|Z\|^2} < \infty$ it follows, that for any $s \in [0,1]$, $t \to \sigma(s,t)$ is right-continuous.
This finishes the proof of (\ref{rn5}). The equivalence with (\ref{rn4}) is obvious.
\end{proof}
For $k \in \N$ let $\Dy_k = \{i 2^{-k}: i = 1, \ldots, 2^k-1\}$ and $\Dy = \bigcup_{k \geq 1} \Dy_k$ be the set of dyadic numbers on $(0,1)$. For $t \in [0,1)$ and a c{\`a}dl{\`a}g  function $f$, we define $f(t-) = \lim_{s \uparrow t} f(s)$ and
$\Delta f(t) = f(t) - f(t-)$. Then, as $Z$ is almost surely c{\`a}dl{\`a}g, the previous theorem also implies
\begin{align} \label{dist:delta}
\Law( \Delta Z(t)) = \mathcal{N}\left(0,   \gamma \kappa^{i} \right)\end{align}
for any $t \in \Dy$ where $i$ is minimal with $t \in \Dy_i$. Here and subsequently, $\mathcal{N}(\mu, \sigma^2)$ denotes the normal distribution with mean $\mu$ and variance $\sigma^2$.
\begin{cor}
 Almost surely, $Z$ is continuous at $t$ for all $t \notin \Dy$. On the contrary, for any $t \in \Dy$, almost surely, $Z$ is not continuous at $t$.
\end{cor}
\begin{proof} Let $A$ be a set of measure one such that $Z_n \to Z$ uniformly on $A$. As $Z_n$ is continuous at $t$ for all $n$ if $t \notin \Dy$ it follows that $Z$ is continuous at $t$ for all $t \notin \Dy$ on $A$, thus almost surely. For $t \in \Dy$, discontinuity follows immediately from \eqref{dist:delta}.
\end{proof}


More refined path properties are discussed in Sections \ref{rn_sec41} and \ref{rn_sec42}.   Simulations of realizations of $Z_{10}$  for $\alpha = 1/2$ are presented below in Figure \ref{fig_simu} to indicate the
structure  of the paths of the limit process $Z$.

\begin{figure}[h]
\includegraphics[scale=0.6]{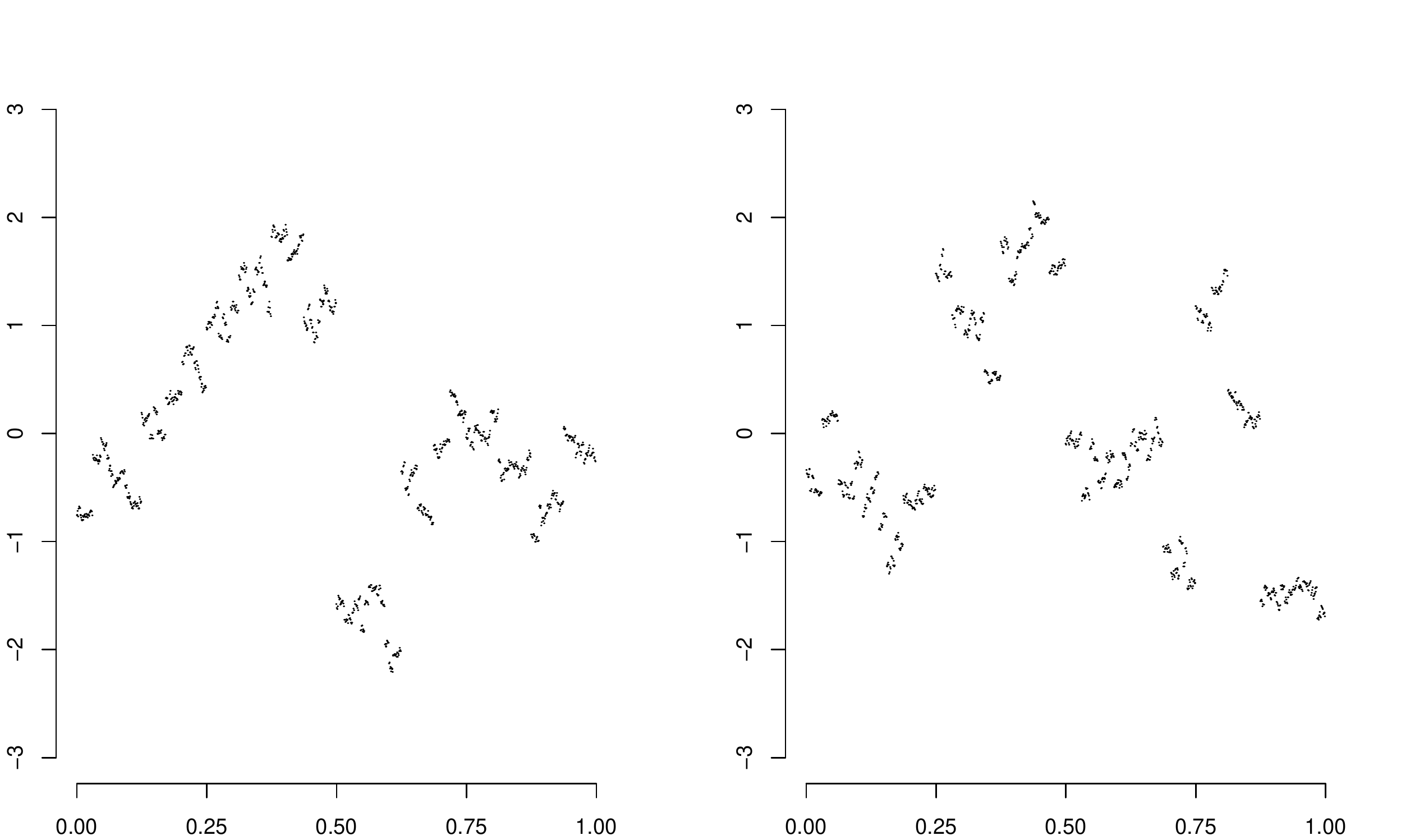}
\caption{Simulations of two independent realizations of $Z_{10}$ for $\alpha = 1/2$.}
\label{fig_simu}
\end{figure}

\section{Analysis of the  Quickselect process}\label{rn_sec3}
Our asymptotic analysis to prove the functional limit laws for  the processes in Theorems \ref{rn_thm_in1} and \ref{rn_thm_in2} is organized as follows. In Section \ref{secc_31} we state a recurrence relation on which the whole analysis is based. To apply ideas from the contraction method we need to derive a distributional fixed point equation for a potential limit of the normalized processes as captured by the map $T$ in (\ref{rn_mapT}). For this, in Section \ref{secc_32} first the asymptotic behavior of the size $I_n$ of $S_\le$ is identified. Then in Section \ref{secc_33} a recurrence for the normalized processes appearing in Theorem \ref{rn_thm_in1} is derived. The random quantities are all embedded on one probability space and coupled in such a way that distances can be
bounded pointwise (with respect to randomness $\omega$) in the supremum norm on $\Do$. We keep the jumps of a couple of auxiliary processes exactly aligned to those of $Y_n$ in order to be able to bound distances by contraction arguments. The necessary  deformations  in time to align with the jumps of the limit process $Z$ are afterwards done  in Proposition \ref{prop3}.

\subsection{Preliminaries}\label{secc_31}
Our analysis is based on a recurrence for the distributions of the processes $X_n^{(2)}= (X_n^{(2)}(\ell))_{1\le \ell\le n}$ and $X_n^{(3)}= (X_n^{(3)}(\ell))_{1\le \ell\le n}$.
 Note that after the selection of the median from the subset the $k$ elements of the subset can already be assigned to the sets $S_<$, $S_>$ and $S_\le$ respectively so that only $n-k$ remaining elements need to be compared with the pivot element. We denote  the rank of the pivot element chosen in the first step  by $I_n$. We set
 $X_0^{(2)}=X_0^{(3)}:=0$. Then we have $X_1^{(2)}=X_1^{(3)}=0$ and, for all $n\ge 2$,
\begin{align}
X^{(3)}_n \stackrel{d}{=} \left( \I{\ell < I_n} X^{(3)}_{I_n-1}(\ell)  + \I{\ell \geq I_n +1} \widehat X^{(3)}_{n-I_n}(\ell - I_n) + n-k + T_n\right)_{1\le \ell\le n},  \label{rec3ver}
\end{align}
where $T_n$,$I_n$, $X^{(3)}_{0},\ldots, X^{(3)}_{n-1}$, $\widehat X^{(3)}_{0},\ldots, \widehat X^{(3)}_{n-1}$ are independent and $\widehat{X}^{(3)}_{j}$ is distributed as $X^{(3)}_{j}$ for $0\le j \le n-1$. The stated independence is satisfied since in subsequent partitioning steps all choices of subsets are made independently.
For the $2$-version we have the same initial values as for the $3$-version and, for all $n\ge 2$ that
\begin{align} X^{(2)}_n \stackrel{d}{=} \left( \I{\ell < I_n+1} X^{(2)}_{I_n}(\ell)  + \I{\ell \geq I_n +1} \widehat X^{(2)}_{n-I_n}(\ell - I_n) + n-k + T_n\right)_{1\le \ell\le n}, \label{rec2ver}
\end{align}
with conditions on independence and identical distributions analogous to the $3$-version in (\ref{rec3ver}).

Recall that $T_n$ is the number of key comparisons for the identification of the median within the random subset and that we assume condition (\ref{rn_con_med}).

We choose $n_0$ large enough such that $k(n)\ge 3$ for all $n\ge n_0$. This ensures that $I_n < n$ for all $n\ge n_0$. 

\subsection{Asymptotics for the pivot and sublist sizes}
\label{secc_32}
 For simplicity of representation, we assume $c=1$, i.e.\ $k =k(n) \sim n^\alpha$ with $\alpha \in (0,1/2]$ for the remainder of the section. 

Elements in the presample of size $k$ are chosen without replacement, thus the distribution of $I_n$ is given by 
\begin{align} \label{weights} \Prob{I_n= i} = \frac{ {i-1 \choose (k-1)/2} {n-i \choose (k-1)/2}}{{n \choose k}}, \quad \frac{k+1}{2} \leq i \leq n-\frac{k-1}{2}. \end{align}
 Equivalently,
$$\Law(I_n) = \Law \left(\frac{k+1}{2} + \text{Bin} \left(n-k, \text{Beta} \left(\frac{k+1}{2}, \frac{k+1}{2} \right) \right)\right),$$
where, here and subsequently, for $n \in \N, p \in [0,1]$, $\text{Bin}(n,p)$ denotes a random variable with the binomial distribution for $n$ trials with success probability $p$. Moreover, for $\alpha, \beta > 0$,  $\text{Beta}(\alpha, \beta)$ denotes a random variable with the beta distribution with parameters $\alpha, \beta$.

Subsequently, let $(M_n)_{n\ge 1}$ be a sequence of random variables with   the beta distribution with parameters $(k+1)/2,(k+1)/2$.
\begin{lem}\label{rn_lem_mn}
 We have
 \begin{align*}
  \E {M_n} = \frac{1}{2}, \quad \Var(M_n) = \frac{1}{4 (k+2)} \sim \frac{1}{4} n^{-\alpha},
 \end{align*}
and, for $n \to \infty$,
\begin{align*}
n^{\alpha/2}\left(M_n  - \frac{1}{2}\right) \stackrel{d}{ \longrightarrow} \mathcal{N} \left(0,\frac{1}{4}\right).
\end{align*}
\end{lem}
\begin{proof} The  expressions for mean and variance follow by straightforward calculations. For the  limit theorem note that for the beta distribution and the binomial distribution we have the following identity
\begin{align} \label{verb.binbeta}
\Prob{\text{Beta}(a,b) < x} = \Prob{\text{Bin}(a+b-1, x) \geq a} \end{align}
 for all $a,b\in\N$ and $x\in(0,1)$. Applying this to $M_n$ and using the central limit theorem, e.g., in the version of de Moivre-Laplace implies the assertion.
\end{proof}
For the size $I_n$ of the left sublist generated in the first partitioning step we have:
\begin{lem} \label{lem:split}
 We have
 \begin{align*}
 \E{I_n} = \frac{n+1}{2}, \quad \Var(I_n) = \frac{1}{4}\left(k (2n-k) + n\left(\frac{n-1}{k+2}+1\right)\right) \sim \frac{1}{4} n^{2-\alpha}. 
\end{align*}
and
\begin{align*}
 \frac{I_n - n/2}{n^{1-\alpha/2}} \stackrel{d}{\longrightarrow} \mathcal{N} \left(0,\frac{1}{4}\right).
\end{align*}
\end{lem}
\begin{proof}
The first two moments follow from Lemma \ref{rn_lem_mn}.
Given $M_n$, let $X_n$ have the binomial distribution with parameters $n-k, M_n$ and set $I_n = \frac{k+1}{2} + X_n$.
By Skorokhod's representation theorem, we may assume  the existence of a sequence $(\FF_n)$, where $\FF_n$ has the distribution  of
$n^{\alpha/2}(M_n  - 1/2)$ such that $\FF_n \to \mathbf{N}$ almost surely where $\mathbf{N}$ has the normal $\mathcal{N}(0,\frac{1}{4})$ distribution.
Let $\MM_n = \FF_n n^{-\alpha/2} + 1/2$ and
construct $\XX_n$ and $\II_n$ such as $X_n$ and $I_n$ but based on the $\MM_n$.
Decomposition yields
\begin{align*}
 \frac{\II_n - n/2}{n^{1-\alpha/2}} = &\frac{\XX_n - \MM_n(n-k)}{\sqrt{(n-k)\MM_n(1-\MM_n)}} \frac{\sqrt{(n-k)\MM_n(1-\MM_n))}}{n^{1-\alpha/2}} \\
&~+  \frac{n \MM_n - n/2}{n^{1-\alpha/2}} - \frac{k \MM_n - k/2 + 1/2}{n^{1-\alpha/2}}.
\end{align*}
By construction, the second summand of the latter display tends to $\mathbf N$ almost surely. Moreover, the third summand tends to zero almost surely.
By conditioning on $(\MM_n)$ and the fact that $\MM_n \to 1/2$ almost surely, the first factor of the first summand converges to a standard normal distribution by the central limit theorem for sums of independent and uniformly bounded random variables. As the second factor of the first summand tends to zero almost surely, the first summand converges to zero in probability. This shows
$$ \frac{\II_n - n/2}{n^{1-\alpha/2}} \to \mathbf{N}, \quad (n \to \infty)$$
in probability.  \end{proof}
 More refined information about the distribution of $M_n$ is given in the Appendix.

\subsection{Proof of Theorems \ref{rn_thm_in1} and \ref{rn_thm_in2}} \label{secc_33}
We first discuss the 2-version of the process and
 recall the normalization from Theorem \ref{rn_thm_in1} which we denote by $Y_0:=0$ and
 \begin{align*}
 Y_n(t) := \frac{X_n^{(2)}(\lfloor tn \rfloor + 1) - 2n}{n^{1-\alpha/2}},\quad t\in [0,1], n\ge 1,
\end{align*}
with the convention $X_n^{(2)}(n+1):=X_n^{(2)}(n)$.
Then, $Y_n:=(Y_n(t))_{t\in[0,1]}$ satisfies, as a random variable in $(\Do,d_{sk})$, for $n\ge n_0$ that
\begin{align*}
Y_n&  \stackrel{d}{=}  \left( \frac{I_n}{n} \right)^{1-\alpha/2} \mathfrak{A}_\frac{I_n}{n}\left(Y_{I_n}\right) +  \left( \frac{n-I_n}{n} \right)^{1-\alpha/2} \mathfrak{B}_\frac{I_n}{n}\left(\widehat Y_{n-I_n}\right)  \\
 &\quad~ + \frac{1}{n^{1-\alpha/2}} \left(T_n -k +  \I{t < I_n /n} (2I_n - n) +  \I{t \geq I_n /n} (n -2I_n) \right)
\end{align*} on $(\Do, d_{sk})$ with conditions on independence and distributional copies as in (\ref{rec3ver}).\\

Now, we embed all the relevant random variables on one probability space such that we have appropriate almost sure
convergences. Throughout we use boldface characters to denote the embedded quantities. To be specific,
by Skorokhod's representation theorem and Lemma \ref{lem:split}, we can construct a set of independent and identically distributed random variates $\{(\pS_n^\vartheta)_{n \geq n_0}, \NN_\vartheta, \vartheta \in \Theta\}$ such that $\NN_\vartheta$ has the standard normal distribution, $\pS_n^\vartheta$ has the distribution of $(2I_n - n)/n^{1-\alpha/2}$ and $ \pS_n^\vartheta \to \NN_\vartheta$ almost surely. Moreover,  by Lemma \ref{lem:split}, we have
$$\E{| \pS_n^\vartheta - \NN_\vartheta|^s} \to 0, \quad n \to \infty $$ for any $1 \leq s \leq 2$. Furthermore, note that $\Law(I_n) = \Law(\JJ_n^\vartheta)$ where
 $\JJ_n^\vartheta := \pS_n^\vartheta \cdot n^{1-\alpha/2}/2 + n/2$. We can further augment this set of random variables by another set $\{ \TT^\vartheta_n, n \geq n_0,  \vartheta \in \Theta\}$ of independent random variables, independent of $(\pS_n^\vartheta)_{n \geq 0}, \NN_\vartheta, \vartheta \in \Theta,$ such that  $\Law(\TT^\vartheta_n) = \Law( T_n)$.
Let $\YY^\vartheta_0 := 0$ and $\{\YY_i^\vartheta\,|\, i < n_0, \vartheta \in \Theta\}$ be a set of independent processes with $\Law(\YY_i^\vartheta) = \Law(Y_i)$, also independent of the family of random variables defined above. For $n \geq n_0$, we define recursively
\begin{align*}
\YY^\vartheta_n & := \left( \frac{\JJ_n^\vartheta}{n} \right)^{1-\alpha/2} \mathfrak{A}_\frac{\JJ_n^\vartheta}{n} \left( \YY^{\vartheta0}_{\JJ_n^\vartheta}\right) +   \left( \frac{n-\JJ_n^\vartheta}{n} \right)^{1-\alpha/2} \mathfrak{B}_\frac{\JJ_n^\vartheta}{n}\left(  \YY^{\vartheta1}_{n-\JJ_n^\vartheta}\right)  \\
 &\quad~+ \frac{1}{n^{1-\alpha/2}} \left(\TT^\vartheta_n -k+ \I{t < \JJ_n^\vartheta /n} (2\JJ_n^\vartheta - n) +   \I{t \geq \JJ_n^\vartheta /n} (n -2\JJ_n^\vartheta) \right).
\end{align*}
 By construction, we have $\Law(\YY^\vartheta_n) = \Law(Y_n)$ for all $n \in \N$, since the sequences $(\YY^\vartheta_n)_{n\ge 0}$ and $(Y_n)_{n\ge 0}$
satisfy the same distributional recurrence and have the same initial distributions for $i = 0, \ldots, n_0-1$.
Subsequently, we use the sets $\{Z_n^\vartheta, n \in \N_0, \vartheta \in \Theta\}$ and  $\{Z^\vartheta, \vartheta \in \Theta \}$ as defined in (\ref{recdefproc}) and Lemma \ref{rnlem} where the construction is executed using the particular set of random variables $\{\NN_\vartheta, \vartheta \in \Theta\}$. We denote the resulting random variables by $\ZZ_n^\vartheta, n \in \N, \vartheta \in \Theta$ and $\ZZ^\vartheta$, $\vartheta \in \Theta$.

To start bounding distances between $\YY_n$ and $\ZZ_n$
  we use two intermediate sequences of stochastic processes $\QQ^\vartheta_n$ and $\RR^\vartheta_n$   in $(\Do,d_{sk})$. First, let $\QQ^\vartheta_i := 0$ for all $\vartheta \in \Theta, i < n_0$ and,  recursively for all $n\ge n_0$,
\begin{align*}
\QQ^\vartheta_n & := \left( \frac{\JJ_n^\vartheta}{n} \right)^{1-\alpha/2} \mathfrak{A}_\frac{\JJ_n^\vartheta}{n} \left(\QQ^{\vartheta0}_{\JJ_n^\vartheta}\right) +   \left( \frac{n-\JJ_n^\vartheta}{n} \right)^{1-\alpha/2} \mathfrak{B}_\frac{\JJ_n^\vartheta}{n}\left( \QQ^{\vartheta1}_{n-\JJ_n^\vartheta}\right)  \\
 &\quad~+ \I{t < \JJ_n^\vartheta /n} \NN_\vartheta -   \I{t \geq \JJ_n^\vartheta /n} \NN_\vartheta .
\end{align*}
Second, $\RR^\vartheta_i := 0$ or all $\vartheta \in \Theta, i < n_0$ and,  recursively for all $n\ge n_0$,
\begin{align}
\RR^\vartheta_n & :=  \left( \frac 1 2\right)^{1-\alpha/2}\mathfrak{A}_\frac{\JJ_n^\vartheta}{n} \left( \RR^{\vartheta0}_{\JJ_n^\vartheta}\right)+ \left( \frac 1 2 \right)^{1-\alpha/2} \mathfrak{B}_\frac{\JJ_n^\vartheta}{n}\left( \RR^{\vartheta1}_{n-\JJ_n^\vartheta}\right) \label{ren_def_rr}\\
 &\quad~ +  \I{t < \JJ_n^\vartheta /n} \NN_\vartheta -   \I{t \geq \JJ_n^\vartheta /n} \NN_\vartheta.\nonumber
\end{align}
The proof of the functional limit law in Theorem \ref{rn_thm_in1} is organized by splitting the difference between $\YY_n^\epsilon$ and $\ZZ^\epsilon$ into several intermediate
differences involving the terms defined above.  As in Definition \ref{rn_nota} we use the abbreviations $\YY_n:=\YY_n^\epsilon$, $\QQ_n:=\QQ_n^\epsilon$, $\RR_n:=\RR_n^\epsilon$ and $\ZZ_n:=\ZZ_n^\epsilon$.
\begin{prop} \label{prop1}
As $n \to \infty$, we have $\E{\| \YY_n - \QQ_n\|^2} \to 0$.
\end{prop}
\begin{prop} \label{prop2}
As $n \to \infty$, we have $\E{\| \QQ_n - \RR_n\|^2} \to 0$.
\end{prop}
\begin{prop} \label{prop3}
As $n \to \infty$, we have $d_{sk}(\RR_n, \ZZ_n) \to 0$ in probability.
\end{prop}
These three propositions immediately 
yield $d_{sk}(\YY_n, \ZZ) \to 0$ in probability and thus Theorem \ref{rn_thm_in1}. From this Theorem \ref{rn_thm_in2} follows from Theorem \ref{rn_thm_in1} and Lemma \ref{lem3}.  The proof of Theorem \ref{thm_sup} is given at the
end of this section.  Corollary \ref{cormain2} gives additional information. Here, for the sake of completeness, we formulate with a general the parameter $c > 0$ as in
Theorem \ref{rn_thm_in1} and Theorem \ref{rn_thm_in2}.
\begin{cor} \label{cormain2}
Let $t \notin \Dy$.
If $t_n \to t$ then $\YY_n(t_n) \to \ZZ(t)$ in probability with convergence of all moments. Thus, for all $(\ell_n)_{n\ge 1}$ with $\ell_n\in\{1,\ldots,n\}$ and $\ell_n / n \to t$ we have $$\frac{X_n^{(2)}(\ell_n) - 2n}{n^{1-\alpha/2}/\sqrt{c}} \stackrel{d}{\longrightarrow} \mathcal{N}\left(0,\frac{1}{1-\kappa}\right) \quad (n\to\infty)$$
together with convergence of all moments. The same is true for the $3$-version $X_n^{(3)}$.
\end{cor}
The rest of this section contains the proofs of our statements.
\begin{proof}[Proof of Proposition \ref{prop1}]
By construction, we have
\begin{align}
\lefteqn{(\YY_n(t) - \QQ_n(t))^2} \nonumber\\
&= \left( \frac{\JJ_n}{n} \right)^{2-\alpha} \left(\mathfrak{A}_\frac{\JJ_n}{n} \left(\YY^{(0)}_{\JJ_n} - \QQ^{(0)}_{\JJ_n}\right)\right)^2  + \left( \frac{n - \JJ_n}{n} \right)^{2-\alpha} \left(\mathfrak{B}_\frac{\JJ_n}{n} \left(\YY^{(1)}_{ n-\JJ_n} - \QQ^{(1)}_{n-\JJ_n}\right)\right)^2 \nonumber\\
&\quad~+ \frac{(\TT_n-k)^2}{n^{2-\alpha}} +  \I{t < \JJ_n /n}(\pS_n - \NN)^2 +   \I{t \geq \JJ_n /n} (\pS_n - \NN)^2\label{rn_su1}\\
&\quad~+ 2 \frac{\TT_n-k}{n^{1-\alpha/2}} \left( \I{t < \JJ_n /n}(\pS_n - \NN) -  \I{t \geq \JJ_n /n} (\pS_n - \NN)\right)\label{rn_su2} \\
&\quad~+ 2\left( \frac{\TT_n-k}{n^{1-\alpha/2}} + \I{t < \JJ_n /n}(\pS_n - \NN) -   \I{t \geq \JJ_n /n} (\pS_n - \NN) \right)
\label{rn_su3}\\
&\quad \quad\quad \times \left( \left( \frac{\JJ_n}{n} \right)^{1-\alpha/2} \mathfrak{A}_\frac{\JJ_n}{n} \left(\YY^{(0)}_{\JJ_n} - \QQ^{(0)}_{\JJ_n}\right)
 + \left( \frac{n - \JJ_n}{n} \right)^{1-\alpha/2} \mathfrak{B}_\frac{\JJ_n}{n} \left(\YY^{(1)}_{ n-\JJ_n} - \QQ^{(1)}_{n-\JJ_n}\right) \right)  \nonumber
 \end{align}
We now take the supremum over $t \in [0,1]$ and the expectation on both sides. Then, by construction, the summands in lines (\ref{rn_su1}) and (\ref{rn_su2}) vanish as $n \to \infty$. Using the Cauchy-Schwarz inequality for the product in (\ref{rn_su3}) and furthermore $\|\mathfrak{A}_u\|=\|\mathfrak{B}_u\|=1$ we obtain altogether that
\begin{align}
\lefteqn{\E{\|\YY_n - \QQ_n\|^2}} \nonumber \\
& \leq   \E{ \left( \frac{\JJ_n}{n} \right)^{2-\alpha} \left\|  \YY^{(0)}_{\JJ_n} - \QQ^{(0)}_{\JJ_n}\right\|^2} + \E{ \left( \frac{n - \JJ_n}{n} \right)^{2-\alpha} \left\|\YY^{(1)}_{ n-\JJ_n} - \QQ^{(1)}_{n-\JJ_n}\right\|^2} \nonumber \\
&\quad~+ \varepsilon_n  \left( \E{ \left( \frac{\JJ_n}{n} \right)^{2-\alpha} \left\|  \YY^{(0)}_{\JJ_n} - \QQ^{(0)}_{\JJ_n}\right\|^2} + \E{ \left( \frac{n - \JJ_n}{n} \right)^{2-\alpha} \left\|\YY^{(1)}_{ n-\JJ_n} - \QQ^{(1)}_{n-\JJ_n}\right\|^2} \right)^{1/2}\nonumber \\
&\quad~+ \varepsilon'_n, \label{eq0}
\end{align}
where $\varepsilon_n, \varepsilon'_n \to 0$. Now, the arguments to infer $\E{\|\YY_n - \QQ_n\|^2} \to 0$ are standard in the framework of the contraction method. In a first step, one shows that the sequence $\Delta_n := \E{\|\YY_n - \QQ_n\|^2}$ is bounded. To this end, assume that $\Delta_m \leq C$ for all $m < n$ with $C \geq 1$. Then, the last display implies
\begin{align*}
\Delta_n \leq C \Bigg( &\E{ \left( \frac{\JJ_n}{n} \right)^{2-\alpha} +  \left( \frac{n - \JJ_n}{n} \right)^{2-\alpha}} \\
&~+ \varepsilon_n \left(\E{ \left( \frac{\JJ_n}{n} \right)^{2-\alpha} +  \left( \frac{n - \JJ_n}{n} \right)^{2-\alpha}}\right)^{1/2} \Bigg) + \varepsilon'_n.
\end{align*}
As $\lim_{n \to \infty} \E{ \left( \frac{\JJ_n}{n} \right)^{2-\alpha} +  \left( \frac{n - \JJ_n}{n} \right)^{2-\alpha}} = 2^{-(1-\alpha)}<1$, we can
deduce $\Delta_n \leq C$ for all sufficiently large  $n$. Then, one shows that $\limsup_{n \to \infty} \Delta_n = 0$ as follows. Start with
denoting $D = \sup_{n \geq 0} \Delta_n$ and $\beta = \limsup_{n \to \infty} \Delta_n.$ Let $\delta > 0$ be arbitrary and $\ell$ large enough such that
$\Delta_n \leq \beta + \delta$ and $\E{ \left( \frac{\JJ_n}{n} \right)^{2-\alpha} +  \left( \frac{n - \JJ_n}{n} \right)^{2-\alpha}} \leq 2^{-(1-\alpha)} + \delta$ for all $n\geq \ell$. Moreover, we can assume $n$ to be large enough to satisfy $\Prob{\ell \leq \JJ_n \leq n -\ell} \geq 1- \delta$. Then, \eqref{eq0} implies
\begin{align*}
\Delta_n \leq D \delta + (\beta + \delta)(2^{-(1-\alpha)} + \delta) + \varepsilon_n \left(D \delta + (\beta + \delta)(2^{-(1-\alpha)} + \delta)\right)^{1/2} + \varepsilon'_n
\end{align*}
Taking the limit superior on both sides and then letting $\delta \downarrow 0$ shows $\beta \leq 2^{-(1-\alpha)} \beta$. Thus, $\beta = 0$.
\end{proof}
\begin{proof}[Proof of Proposition \ref{prop2}]
By definition, we have
\begin{align*}
\lefteqn{\|\QQ_n - \RR_n\|}\\
& \leq \left \| \left( \left( \frac{\JJ_n}{n} \right)^{1-\alpha/2} - \left( \frac 1 2 \right)^{1-\alpha/2} \right) \mathfrak{A}_\frac{\JJ_n}{n}\left( \RR^{(0)}_{\JJ_n}\right) \right.\\
 &\quad \left.~+ \left( \left( \frac{n- \JJ_n}{n} \right)^{1-\alpha/2} - \left( \frac 1 2 \right)^{1-\alpha/2} \right) \mathfrak{B}_\frac{\JJ_n}{n}\left( \RR^{(1)}_{n-\JJ_n}\right)  \right \| \\
&\quad ~+ \left \| \left( \frac{\JJ_n}{n} \right)^{1-\alpha/2} \mathfrak{A}_\frac{\JJ_n}{n} \left(  \RR^{(0)}_{\JJ_n} - \QQ^{(0)}_{\JJ_n} \right) +\left( \frac{n- \JJ_n}{n} \right)^{1-\alpha/2} \mathfrak{B}_\frac{\JJ_n}{n} \left( \RR^{(1)}_{ n-\JJ_n} - \QQ^{(1)}_{n-\JJ_n}\right) \right \|
\end{align*}
Let $\varepsilon''_n$ be the second moment of the first summand in the latter display. By construction, we have $\| \RR_n \| \leq \| \ZZ_n \|$ for all $n \in \N$. Thus, Lemma \ref{rnlem} implies that the sequence  $\E{\| \RR_n\|^2}$ is  bounded.
Using the Cauchy-Schwarz inequality, we infer that $\varepsilon''_n \to 0$ as $n \to \infty$. Yet another application of the Cauchy-Schwarz inequality shows
\begin{align*}
\lefteqn{\E{\|\QQ_n - \RR_n\|^2}} \\
 &\leq \E{  \left( \frac{\JJ_n}{n} \right)^{2-\alpha} \left\|  \QQ^{(0)}_{\JJ_n} - \RR^{(0)}_{\JJ_n}\right\|^2} + \E{ \left( \frac{n- \JJ_n}{n} \right)^{2-\alpha} \left\|\QQ^{(1)}_{ n-\JJ_n} - \RR^{(1)}_{n-\JJ_n}\right\|^2} \\
&\quad~+ \sqrt{\varepsilon''_n} \left( \E{  \left( \frac{\JJ_n}{n} \right)^{2-\alpha} \left\|  \QQ^{(0)}_{\JJ_n} - \RR^{(0)}_{\JJ_n}\right\|^2} + \E{ \left( \frac{n- \JJ_n}{n} \right)^{2-\alpha} \left\|\QQ^{(1)}_{ n-\JJ_n} - \RR^{(1)}_{n-\JJ_n}\right\|^2} \right)^{1/2} \\
&\quad~+ \varepsilon''_n.
\end{align*}
The result now follows by an argument similar to the proof of Proposition \ref{prop1}.
\end{proof}
\begin{proof}[Proof of Proposition \ref{prop3}]
Let $\varepsilon >0$. By Lemma \ref{rnlem} there exists  an $n_1\in\N$  such that
$$\Prob{ \sup_{n \geq n_1} \|\ZZ_n - \ZZ_{n_1}\| > \varepsilon} \leq \varepsilon.$$
Let $n \geq n_1$. When applying the recurrence (\ref{ren_def_rr}) for $\RR_n$ iteratively $n_1$ times we obtain a representation of $\RR_n$ with at most $2^{n_1}$ summands. Each summand corresponds to one of the $2^{n_1}$ sublists (some possibly being empty) generated by the algorithm in the first $n_1$ recursive steps.  Let $A_n$ denote  the event that  each of these $2^{n_1}$ sublists has size at least $n_0$. On $A_n$ the split into these first $2^{n_1}$ sublists causes $2^{n_1}-1$ points of discontinuity of $\RR_n$ which we denote by $0<T_n^1<T_n^2<\cdots<T_n^{2^{n_1}-1}$. In fact, in general $\RR_n$ has additional points of discontinuity caused by splits when further unfolding the recurrence (\ref{ren_def_rr}). Moreover, we denote the points of discontinuity of $\ZZ_{n_1}$ by $\tau_n^k = k /2^{{n_1}}$ for $k = 1, \ldots, 2^{n_1} -1$.

By Lemma \ref{lem:split} we have
$\JJ_n^\vartheta / n \to 1/2$ for each $\vartheta \in \Theta$ almost surely, hence
\begin{align}\label{rn_event}
\Prob{A_n \cap \bigcap_{k = 1}^{2^{n_1}-1} \{|T_n^k - \tau_n^k| < \varepsilon\}} \to 1, \quad (n\to\infty).
\end{align}
To bound the Skorokhod distance  between $\RR_n$ and $\ZZ_n$ we define a deformation of time as follows:
On $A_n$ let $\lambda_n : [0,1] \to [0,1]$ be defined by $\lambda_n(0) :=0, \lambda_n(1) := 1, \lambda_n(\tau_n^k) = T_n^k$ for  $k = 1, \ldots, 2^{n_1}-1$ and linear in between these points. Then, with $\mathrm{id}$ the identity $t\mapsto t$ on $[0,1]$ we have on the event in (\ref{rn_event}) that $\|\lambda_n - \mathrm{id}\|<\varepsilon$.  This implies for all $n\ge n_1$ that
$$\bigcap_{m \geq n_1} \{\|\ZZ_m  - \ZZ_{n_1}\| < \varepsilon\} \cap A_n \cap \bigcap_{k = 1}^{2^{n_1}-1} \{|T_n^k - \tau^k| < \varepsilon\} \subseteq \{d_{sk}(\RR_n,\ZZ_n) \leq 2 \varepsilon\}.$$
To see  this, note that on event on the left hand side, we
 have $\|\lambda_n- \mathrm{id}\| \le \varepsilon$ and
$$\|\RR_n \circ\lambda_n - \ZZ_n\| \leq \|\RR_n\circ\lambda_n - \ZZ_{n_1}\| + \|\ZZ_{n_1}- \ZZ_{n}\|\le 2\varepsilon.$$
Thus, for all $n$ sufficiently large,
$\Prob{d_{sk}(\RR_n, \ZZ_n) \leq 2 \varepsilon} \geq 1- 2 \varepsilon$.
\end{proof}
\begin{proof}[Proof of Corollary \ref{cormain2}]
Let $t \in [0,1]\setminus \Dy$ and $(t_n)_{n\ge 1}$ a sequence in $[0,1]$ with $t_n \to t$. By Proposition \ref{prop1} we have $\E{|\YY_n(t_n) - \QQ_n(t_n)|^2} \to 0$ as $n\to\infty$. Moreover, $d_{sk}(\QQ_n, \ZZ) \to 0$ in probability by Propositions \ref{prop2} and \ref{prop3}. As  $\ZZ$ is almost surely continuous at $t$, it follows that
$\YY_n(t_n) \to \ZZ(t)$ in probability. Based on the uniform boundedness of the sequence $\E{\|\YY_n\|^2}$ a simple induction relying on its recursive definition shows that $\sup_{n\ge 1}\E{\|\YY_n\|^m}<\infty$  for all
$m \in \N$. This implies the result for the 2-version. The statement about the 3-version follows from this and Lemma \ref{lem3}.
\end{proof}
\begin{proof}[Proof of Theorem \ref{thm_sup}]
Distributional convergence for the $2$-version follows directly from Theorem \ref{rn_thm_in1}. The proof of Theorem \ref{rn_thm_in1} has also revealed that
$\|Z\|$ has finite moments of all orders and that the sequences $(\|\YY_n\|)_{n \geq 1}$  and  $(\|\QQ_n\|)_{n \geq 1}$ are both bounded in $L_p$ for any $1\le p <\infty$. This shows the claim of Theorem \ref{thm_sup} for the
$2$-version. An alternative approach which works for both the $2$- and the $3$-version relies on the contraction method for $\max$-type recurrences.
This  is based on the distributional recurrence $$
W_n \stackrel{d}{=} \max(W_{I_n-1}, \widetilde W_{n-I_n}) + n -k + T_n,$$
where $W_n := \sup_{1 \leq \ell \leq n} X_n^{(3)}(\ell)$ and $(\widetilde W_n)_{n \geq 0}$ is an independent copy of $(W_n)_{n \geq 1}$, both independent of
$(I_n,T_n)$. The latter display allows to deduce Theorem \ref{thm_sup} straightforwardly from Theorem 4.6 in \cite{Ru06} together with the characterization of
$\|Z\|$ given in Corollary \ref{cor_sup}.
\end{proof}
\section{Further properties of the limit process}
\label{rn_sec4}

In this section we  first study the supremum of the limit process and derive tail bounds. Then path properties of the limit process $Z$ are
investigated.  Here, first, the variation of the limit process $Z$ is studied. Then, we will endow the unit interval with an alternative metric $d_\kappa$
such that $Z$ has continuous paths with respect to $d_\kappa$. This allows to study  the modulus of continuity and H{\"o}lder continuity properties. In Sections \ref{rn_sec43} and \ref{rn_sec42}, we make use of general results about path continuity and the supremum of Gaussian processes, see, e.g., Adler's book \cite{adler90}, and of the explicit construction of the limit process.

\subsection{The supremum of the limit process} \label{rn_sec43}
Let $S_n^\vartheta = \sup_{t \in [0,1]} Z_n^\vartheta(t)$ and $S^\vartheta = \sup_{t \in [0,1]} Z^\vartheta(t)$.  By the uniform convergence stated in Lemma \ref{rnlem} we have $S_n^\vartheta \rightarrow S^\vartheta$ almost surely. 
The first result concerns a  max-type recurrence for $S_n$ and characterizes the distribution of $S$ as solution of a stochastic fixed-point equation. To this end,
let $\Ms(\R)$ denote the set of probability measures on the real line,
 \begin{align*}
\Ms_p(\R):=\Big\{\mu \in \Ms(\R)\,\Big|\, \int |x|^p\,d\mu(x)<\infty\Big\}, \quad 1\le p<\infty,
 \end{align*}
and $T^* : \Ms(\R) \to \Ms(\R)$ be defined, for $\mu \in\Ms(\R)$, by
\begin{align*}
 T^*(\mu) := {\cal L} \left( (\kappa^{1/2} X_0 + N) \vee (\kappa^{1/2} X_1 - N) \right),
\end{align*}
where $\Law(X_0) = \Law(X_1) = \mu$, $N$ has the standard normal distribution and  $X_0,X_1, N$ are independent, and $\kappa = 2^{\alpha-2}$ (as above).
\begin{cor} \label{cor_sup}
Let $\vartheta \in \Theta$. We have
 \begin{align}
  S^\vartheta_{n+1} &= (\kappa^{1/2} S_{n}^{\vartheta0} + N_\vartheta) \vee (\kappa^{1/2} S_n^{\vartheta1} - N_\vartheta),\quad n\ge 1, \nonumber \\
 S^\vartheta &= (\kappa^{1/2} S^{\vartheta0} + N_\vartheta) \vee (\kappa^{1/2} S^{\vartheta1} - N_\vartheta)\quad \text{almost surely}.  \label{fix:sup1}
 \end{align}
The distribution of $S^\vartheta$ is the unique fixed-point of the restriction of $T^*$ to $\Ms_p(\R)$ for any $p>p_\alpha$ with $p_\alpha$ given in \eqref{const_p}.
\end{cor}
\begin{proof} The  recurrence for $S_n^\vartheta$ and the almost sure identity for $S^\vartheta$ follow by construction and Lemma \ref{rnlem}. The characterization of $\Law(S^\vartheta)$ is a special case of Theorem 3.4 in \cite{neru05}.
\end{proof}
It is a well-known  phenomenon that the supremum of a Gaussian process resembles a Gaussian random variable. This explains the following proposition.
\begin{prop}
For the supremum $S=\sup_{t\in[0,1]}Z(t)$ of the limit process $Z$ from Definition \ref{rn_nota} we have for any $t > 0$ that
\begin{align}
\Prob{|S - \E{S}|  \geq t} \leq 2 \exp\left( -\frac{1-\kappa}{2} t^2 \right)  \label{tail_S}
\end{align}
The same tail bounds are valid when $S$ is replaced by $S_n=\sup_{t\in[0,1]}Z_n(t)$ for any $n \in \N$.  The constant in the exponent on the right hand side of (\ref{tail_S}) is asymptotically optimal as $t \to \infty$
Moreover, we have
\begin{align*}
 \frac{\sqrt{2}}{\sqrt{\pi}(1-\sqrt{\kappa})} \le \E{S} \le \sqrt{\frac{2}{1-2\kappa}},  \quad \Var(S) \leq \frac{1}{1-\kappa}.
\end{align*}
For $\alpha = 1/2$, the first bound leads to $\E{S} \in [1.968\ldots, 2.613\ldots]$.
\end{prop}
\begin{proof}
From Theorem \ref{rn_gau_pro} we have $\Var(Z(t))=1/(1-\kappa)$ for all $t\in[0,1]$.
The tail bound \eqref{tail_S} now follows from a variant of Borell's inequality, see, e.g.\ Theorem 2.1 in \cite{adler90}. For $t \to \infty$, optimality of the constant in the exponent follows directly by replacing $S$ by $Z(t)$. The corresponding bound on $\Var(S)$ can be deduced from Theorem 5.8 in
\cite{boluma13} since there, the assumption of path continuity can be relaxed to regularity. Both results also apply to $S_n$ for $n \in \N$.

For the lower bound on  $\E{S}$ note that there is a $t_0\in[0,1]$ such that the terms $( \I{t_0 \in B_{\vartheta 0}} - \I{t_0 \in B_{\vartheta 1}}) N_\vartheta$ in (\ref{rn_z_rep}) are non-negative for all $\vartheta \in \Theta$. Hence, we obtain $\E{S}\ge \E{Z(t_0)} = \E{|N|}\sum_{i\ge 0} \kappa^{i/2}$, which is the lower bound.

For the upper bound on $\E{S}$ we take squares and expectations on left and right hand side of  (\ref{fix:sup1}). This implies $\E{S^2}\le 2/(1-2\kappa)$ and we obtain the bound from $\E{S}\le \sqrt{\E{S^2}}$.
\end{proof}

\subsection{Variation of paths}
\label{rn_sec41}
We have already seen that the constant $p_\alpha$ defined in \eqref{const_p} is intimately linked to the limit process $Z$. In this section, we will see that this connection extends to path properties of $Z$, more precisely to its path variation.
To formalize the main results of the section we need some notation. For $t \in (0,1]$, let $\Pi(t)$ be the set of all finite decompositions of the interval $[0,t]$.
 Elements $\pi \in \Pi(t)$ we write as  $\pi = \{\tau_1, \tau_2, \ldots, \tau_{k}\}$ with $0 = \tau_1 < \tau_2 < \ldots < \tau_k = t$. We also denote $|\pi| = k$
the size of $\pi$. Moreover, we abbreviate $\text{mesh}(\pi) = \max_{i=1, \ldots, |\pi|-1} |\tau_{i+1} - \tau_i|$.
For a c\`adl\`ag function $f$ and $p > 0, t \in (0,1]$, we define
$$V_{p,t}(f) := \sup_{\pi \in \Pi(t)} \sum_{i=1,\ldots, |\pi|-1} |f(\tau_{i+1}) - f(\tau_i)|^p,$$
where $V_p(f) := V_{p,1}(f)$. Let $N_f$ be the set of discontinuity points of $f$. Then, we set
$$W_{p,t}(f) :=\sum_{s \in N_f \cap [0,t]} |\Delta f(s)|^p,$$
again with $W_p(f) := W_{p,1}(f)$.
Finally, we set
\begin{align} \label{def:ftp}
[f]_t^{(p)} := \lim_{\pi \in \Pi(t) \atop \text{mesh}(\pi) \to 0} \sum_{i = 1, \ldots,  |\pi|-1} |f(\tau_{i+1}) - f(\tau_i)|^p, \end{align}
if the limit exists in $\R^+_0 \cup \{\infty \}$. The c\`adl\`ag property of $f$ implies that, for any $t \in (0,1]$,
\begin{align} \label{vwz}
 V_{p,t}(f) < \infty \Rightarrow W_{p,t}(f) < \infty \quad \text{and} \quad W_{p,t}(f)= \infty \Rightarrow [f]^{(p)}_t = \infty.
\end{align}
The following lemma is well-known in the case $p=1, q = 2$, we did not find a proof for the general case in the literature. Thus, we include one in the Appendix.
\begin{lem} \label{lem:helppq}
 Let $f \in \Do$, $p >0$ and $V_p(f) < \infty$. Then, for any $q > p$, we have
\begin{align} \label{assertionlemma}
 [f]_t^{(q)} = W_{q,t}(f).
\end{align}
 Additionally, the map $t \mapsto [f]_t^{(q)}$ is c\`adl\`ag with $\Delta [f]_t^{(q)} = |\Delta f(t)|^q$.
\end{lem}

The following theorem is the main result of this section. Recall the definition of $p_\alpha$ in (\ref{const_p}) and
\begin{align} \label{def:gamma}
 \gamma = \frac{4-2 \kappa}{1-\kappa}.
\end{align}

\begin{thm} \label{thm:grossp}

\begin{enumerate} \item For  $p > p_\alpha$, we have that, almost surely, $V_p(Z) < \infty$ and
$$[Z]_t^{(p)} = W_{p,t}(Z) = \sum_{s \in \Dy \cap [0,t]} |\Delta Z(t)|^p,$$ where the convergence in \eqref{def:ftp} with $f = Z$ also holds with respect to all moments. For the mean, we have
$$\E{[Z]^{(p)}_t} =  \gamma^{p/2} \E{|N|^p} \sum_{i=1}^\infty \kappa^{pi/2}
\left\lfloor 2^{i-1} t + \frac 12 \right\rfloor.
$$
\item Almost surely, for any $t \in (0,1]$, we have $V_{p_\alpha,t}(Z) = W_{p_\alpha, t}(Z) = [Z]_t^{(p_\alpha)} = \infty$. 
\end{enumerate}

\end{thm}

The proof of the theorem makes use of a simple yet useful tool,  well known, e.g.,  from L\'{e}vy's construction of Brownian  motion.
\begin{lem} \label{lembrown}
Let $c > \sqrt{2 \log 2}$. Then, almost surely, there exists a (random) integer $K$ such that for every $k \geq K$, we have
\begin{align*}
 \sup_{\vartheta \in \Theta: |v| = k} |N_{v}| \leq c \cdot k^{1/2}.
 \end{align*}
\end{lem}
\begin{proof}
We have
\begin{align*}
\lefteqn{ \sum_{k = 1}^{\infty} \Prob {\sup_{\vartheta \in \Theta: |v| = k} |N_{v}| > ck^{1/2} } }\\
&\leq \sum_{k = 1}^{\infty} 2^{k+1} \Prob{N > ck^{1/2}}
 \leq  \sum_{k = 1}^{\infty} \frac{2^{k+1}}{\sqrt{2 \pi}}\int_{ck^{1/2}}^{\infty} y e^{-y^2/2} dy
 =   \sum_{k = 1}^{\infty} \frac{2^{k+1}}{\sqrt{2 \pi}} e^{-c^2k/2} < \infty.
\end{align*}
The  Borel--Cantelli Lemma implies the assertion.
\end{proof}
\begin{proof}[Proof of Theorem \ref{thm:grossp}]
The main part of claim $i)$ follows immediately from Lemma \ref{lem:helppq} upon establishing $V_{p}(Z) < \infty$ for $p > p_\alpha$ almost surely.
To prove this,  let $A$ be a set of measure one and $K = K(\omega)$ for $\omega \in A$ such that the statement of Lemma \ref{lembrown} is satisfied with $c=2$ there.
Let $\pi \in \Pi(1)$. Then, for fixed $\omega \in A$,
\begin{align} \label{b0}
 \sum_{i=1}^{\pi-1} |Z(\tau_{i+1}) - Z(\tau_i)|^p = \sum_{i=1, \ldots, \pi-1 \atop j(\tau_{i+1}, \tau_i) < K} |Z(\tau_{i+1}) - Z(\tau_i)|^p +
\sum_{i=1, \ldots, \pi-1 \atop j(\tau_{i+1}, \tau_i) \geq K} |Z(\tau_{i+1}) - Z(\tau_i)|^p
\end{align}
We will show that both terms on the right hand side can be bounded from above independently of the partition $\pi$. This shows the claim $V_{p}(Z) < \infty$.
The first summand is easier. There are at most $2^\ell$ pairs $(\tau_i, \tau_{i+1})$ such that $j(\tau_i, \tau_{i+1}) = \ell$. Thus,
\begin{align} \label{b1}
\sum_{i=1, \ldots, \pi-1 \atop j(\tau_{i+1}, \tau_i) < K} |Z(\tau_{i+1}) - Z(\tau_i)|^p \leq 2^{K +1} \| Z \|^p.
\end{align}
Next, for $j(\tau_i, \tau_{i+1}) \geq K$,
$$
|Z(\tau_{i+1}) - Z(\tau_i)|^p \le \left(\sum_{\ell = j(\tau_i, \tau_{i+1})}^\infty 4 \kappa^{\ell/2}
\sqrt{\ell }\right)^p
\leq (4D)^p \kappa^{j(\tau_{i+1},\tau_i)p/2} j(\tau_{i+1},\tau_i)^{p/2},
$$
where we have abbreviated
\begin{align} \label{defD} D= \sum_{m=0}^\infty \kappa^{m/2} \sqrt{1+m}.
\end{align}
Summation implies
\begin{align*}
 \sum_{i=1, \ldots, \pi-1 \atop j(\tau_{i+1}, \tau_i) \geq K} |Z(\tau_{i+1}) - Z(\tau_i)|^p & \leq
\sum_{j=K}^{\infty} 2^j (4D)^p \kappa^{jp/2} j^{p/2}
\leq (4D)^p \sum_{j = K}^\infty (2 \kappa^{p/2})^j j^{p/2}.
\end{align*}
Since $2 \kappa^{p/2}<1$ the right hand side of the latter display is finite.
Combining the latter display and \eqref{b1}, we obtain the desired upper bound for \eqref{b0}.
For the convergence of moments let $m \in \N$. Then, for $\pi \in \Pi(1)$, we have
\begin{align*}
 \left \|\sum_{i = 0}^{|\pi|-1} |Z(\tau_{i+1}) - Z(\tau_{i})|^{p}\right\|_m &\leq \sum_{i = 0}^{|\pi|-1}
\||Z(\tau_{i+1}) - Z(\tau_{i})|^p\|_m   \leq \gamma^{p/2} \||N|^{p}\|_m \sum_{k=1}^\infty (2\kappa^{p/2})^k
\end{align*}
The result follows as the last bound does not depend on $\pi$.

Regarding the mean of the $p$-variation, abbreviating $\Dy_0 = \emptyset$, we have
\begin{align*}
\E{[Z]^{(p)}_t}  =  \sum_{s \in \Dy \cap [0,t]} \E{| \Delta Z(s)|^p} & = \gamma^{p/2} \E{|N|^p} \sum_{i=1}^\infty  |\Dy_i \backslash \Dy_{i-1} \cap [0,t] | \kappa^{p i /2}\\
& = \gamma^{p/2} \E{|N|^p}
 \sum_{i=1}^\infty \kappa^{pi/2}
\left\lfloor 2^{i-1} t + \frac 12 \right\rfloor
\end{align*}
which finishes the proof of $i)$.

We move on to the proof of $ii)$. Due to \eqref{vwz} it is sufficient to show that, for any $t \in (0,1]$, we have $W_{p_\alpha,t} = \infty$ almost
surely. Again, we restrict our presentation to the case $t =1$. 
As a warm-up we first investigate the case $p < p_\alpha$.
Let $X_n = \sum_{t \in \Dy_n} |\Delta Z(t)|^{p}$ and $X_n' = \sum_{t \in \Dy_n \backslash \Dy_{n-1}}|\Delta Z(t)|^{p}$. Then
$X_n' \leq X_n \uparrow \sum_{t \in \Dy} |\Delta Z(t)|^{p}$ almost surely. The assertion $\sum_{t \in \Dy} |\Delta Z(t)|^{p} = \infty$ almost surely now follows
easily from Chebychev's inequality and the facts that, as $n\to\infty$,
\begin{align*}
& \E{X_n'} = \sum_{t \in \Dy_n \backslash \Dy_{n-1}}\E{|\Delta Z(t)|^{p}} = \frac 1 2 \gamma^{p/2} \E{|N|^p} (2 \kappa^{p/2})^n \rightarrow \infty, \\
& \Var(X_n') = \sum_{t \in \Dy_n \backslash \Dy_{n-1}}\Var(|\Delta Z(t)|^{p}) = \frac 1 2 \gamma^{p} \Var(|N|^p) (2 \kappa^{p})^n = o\left(\E{X_n'}^2 \right).
\end{align*}
Here, we have used that the random variables $\Delta Z(t)$, $t \in \Dy_n \backslash \Dy_{n-1}$ are independent. Note that this does not
extend to all $t \in \Dy_n$.
The situation is more involved for $p = p_\alpha$. Here, the sequence $(\E{X_n'})$ is constant
which implies
$$\E{X_n}= \frac 1 2  \gamma^{1/(2-\alpha)} \E{|N|^{p_\alpha}}\cdot n.$$
Thus, $\E{[Z]^{(p_\alpha)}_1} = \infty$. The assertion now follows from showing that the variance of $X_n$ grows at most linearly. By definition we have
$$\Var( X_n) = \sum_{t\in \Dy_n} \Var\left( |\Delta Z(t)|^{p_\alpha}\right) + \sum_{s,t \in \Dy_n, s \neq t} \Cov\left(|\Delta Z(t)|^{p_\alpha}, |\Delta Z(s)|^{p_\alpha}\right). $$
First, $$\sum_{t\in \Dy_n} \Var\left( |\Delta Z(t)|^{p_\alpha}\right) \leq \gamma^{p_\alpha} \Var\left(|N|^{p_\alpha}\right) \sum_{k = 0}^\infty (2 \kappa^{p_\alpha})^k,$$
where the right hand side does not depend on $n$. 
For $t\in \Dy_i \backslash \Dy_{i-1}$ and $j > i$, $\Delta Z(t)$ is independent of $\Delta Z(s)$ for all $s \in \Dy_j \backslash \Dy_{j-1}$ except for its direct neighbors. Thus, we have
\begin{align*}
 \lefteqn{\sum_{s,t \in \Dy_n, s \neq t} \Cov\left(|\Delta Z(t)|^{p_\alpha}, |\Delta Z(s)|^{p_\alpha}\right)}  \\
& = 2 \sum_{i=1}^{n} \sum_{t \in \Dy_i \backslash \Dy_{i-1}}  \sum_{s \in \Dy_j \backslash \Dy_{j-1}, \atop i<j\leq n} \Cov\left(|\Delta Z(t)|^{p_\alpha}, |\Delta Z(s)|^{p_\alpha}\right) \\
& \leq 4 \sum_{i=1}^{n} \sum_{t \in \Dy_i \backslash \Dy_{i-1}} \sqrt{\Var( |\Delta Z(t)|^{p_\alpha})} \sqrt{\gamma^{p_\alpha} \Var (|N|^{p_\alpha})} \sum_{j >i} \kappa^{j/(3(1-\alpha))} \\
& \leq 4  \gamma^{p_\alpha} \Var \left(|N|^{p_\alpha}\right)  \sum_{j = 0}^\infty \kappa^{j/(3(1-\alpha))}\cdot n.
\end{align*}
The assertion follows.
\end{proof}

\subsection{Binary topology and path continuity}
\label{rn_sec42}
Regarding path continuity of a Gaussian process $X$ on the unit interval, the canonical choice of a metric is given by $d(s,t) := \sqrt{\E{(X(t)-X(s))^2}}$ for $s,t\in[0,1]$. In our case, that is $X = Z$,
identifying $[0,1]$ with $\{0,1\}^\N$ via the binary representations, $d$ induces the product topology on $\{0,1\}^\N$.
A sequence $(x^{(n)})_{n\ge 1}$ where $x^{(n)}=\sum_{i\ge 1} x^{(n)}_i 2^{-i}$ converges to
$x$ with respect to $d$ if and only if for each $k \in \N$ there exists $n_0 \in \N$ such that $x^{(n)}_i = x_i$ for all $i \leq k$ and $n \geq n_0$.
Convergence $d(x^{(n)}, x) \to 0$ implies $|x^{(n)}-x| \to 0$. Conversely,  $|x^{(n)} -x| \to 0$ implies $d(x^{(n)}, x) \to 0$ if and only if either $x \notin \Dy$ or $x \in \Dy$ and additionally $x^{(n)} \geq x$ for almost all $n$. The limit process $Z$ as well as its $p$-variation for $p > p_\alpha$ are almost surely continuous with respect to $d$.

For notational reasons, we work with an (topologically) equivalent metric: for $x,y \in [0,1]$ with binary representations $x=\sum_{i\ge 1} x_i 2^{-i}$, $y=\sum_{i\ge 1} y_i 2^{-i}$  we define 
\begin{align} \label{defmetrickappa}d_\kappa(x,y) := \kappa^{j(x,y)}. \end{align}
Note again that $d_\kappa$ and $d$ depend on $\alpha$ via $\kappa$.
Finally, working with $d$ or more generally, changing the base in \eqref{defmetrickappa} to any value lower than one will only effect absolute constants in the following results.

\medskip The additive construction of $Z$ somewhat resembles L{\'e}vy's construction of Brownian Motion which guides 
both intuition and proofs in the remainder of this section.

\begin{thm}[Modulus of continuity] \label{thm:modulus}
With $\gamma$ as in (\ref{def:gamma})  
we have, almost surely,
$$ \sqrt{\frac{2 \gamma \log 2}{\log(1/\kappa)}} \leq \limsup_{h \downarrow 0} \sup_{t,s \in [0,1], \atop d_\kappa(t,s) = h} \frac{|Z(t) - Z(s)|}{\sqrt{h \log(1/h)}} \leq \frac{2\sqrt{2 \log 2}}{\sqrt{\log (1/\kappa)}(1-\sqrt{\kappa})},$$
where the $\limsup$ is taken over sequences $h \downarrow 0$ with $h=\kappa^{n}$ for some $n \in \N$.
\end{thm}
\begin{proof}
 We start with the upper bound.
First, let  $\varepsilon > 0$ and $K_1$  be large enough such that $\sum_{j=0}^\infty \kappa^{j/2} \sqrt{1+j/K_1} \leq (1+\varepsilon)/(1-\sqrt{\kappa})$.  Next, let $c > \sqrt{2 \log 2}$ and choose $\omega \in A$ and
$K \geq K_1$ as in Lemma \ref{lembrown}. Let  $h = \kappa^L$ with $L \geq K$. Then, for $t,s \in [0,1]$ with $d_\kappa(t,s) = h$, it follows
\begin{align*}
|Z(t) - Z(s)| & \leq 2c \sum_{j = L}^\infty  \kappa^{j/2} \sqrt{j}  \leq \frac{2c(1+\varepsilon)}{\sqrt{\log (1/\kappa)}(1-\sqrt{\kappa})} \sqrt{h \log (1/h)}.
\end{align*}
Lower bounds follow analogously as for Brownian Motion.
Let $0<v < \sqrt{2 \gamma \log 2}$. For $n \in \N$ and $0 \leq k \leq 2^{n-1}-1$ let
$$A_{k,n} = \{ |Z((2k+1)2^{-n}) - Z(2k\cdot 2^{-n})| > v \sqrt{n} \kappa^{n/2}\}.$$
By construction, for fixed $n \in \N$, the family of events $A_{k,n}, 0 \leq k \leq 2^{n-1}-1$ is independent. Moreover,
\begin{align*}
\Prob{A_{k,n}} = \Prob{|N| \geq \frac{v\sqrt{n}}{ \sqrt{\gamma}} } \ge \Omega\left(\frac{1}{\sqrt{n}}\right)\exp\left(-\frac{v^2 n}{2\gamma}\right).
 \end{align*}
Thus, $2^{n-1} \Prob{A_{k,n}} \to \infty$ as $n \to \infty$. By independence,
$$\Prob{\bigcap_{k=0}^{2^{n-1}-1} A_{k,n}^c} \leq e^{-2^{n-1} \Prob{A_{1,n}}} \to 0.$$
This yields the assertion upon choosing $h = \kappa^n$ for $n$ sufficiently large (and random). 
\end{proof}
Moduli of continuity of the order $\sqrt{h \log(1/h)}$ can also be obtained from general results on Gaussian processes. First, by Theorem 4.6 in \cite{adler90}, which relies on deep results from Talagrand \cite{tala87}, a modulus of continuity is given by
\begin{align*}
\E{\sup_{d_\kappa(s,t) \leq h} (Z_s - Z_t)} & = \sqrt{h} \E{\max \{S_1^*, \ldots, S_{2^n}^*\}}
\end{align*}
where $h = \kappa^n$ and $S_1^*, \ldots, S_{2^n}^*$ are independent random variables, each having the distribution of $\sup_{s,t \in [0,1]} Z(t) - Z(s)$.
An upper bound for  the right hand side in the latter display by use of the bound \eqref{tail_S} leads to a constant
$$\frac{4 \sqrt{2 \log 2}}{\sqrt{1-\kappa}\sqrt{\log 1/\kappa}},$$
which is slightly worse than the upper bound stated in Theorem \ref{thm:modulus}.
Second, the approach towards path continuity relying on the so-called metric entropy of $[0,1]$ with respect to $d_\kappa$ leads to a modulus of continuity of the same order with a random constant, see e.g.\ \cite[Corollary 2.3]{dudley73}.
 \begin{thm} [H{\"o}lder continuity] \label{rn_thm_ho} For any $\beta < 1/2$, almost surely, the paths of $Z$ are H{\"o}lder continuous with exponent $\beta$ with respect to $d_\kappa$.
 For any $\beta > 1/2$, almost surely, the paths of $Z$ are
 nowhere pointwise H{\"o}lder continuous with exponent $\beta$ with respect to $d_\kappa$.
\end{thm}
\begin{proof}
The result for $\beta < 1/2$ follows immediately from the upper bound on the modulus of continuity. Thus, we consider the case $\beta > 1/2$.
We only treat the interval $[0,1)$, the proof for $t=1$ being easier. We adopt the proof of the corresponding statement for the Brownian Motion from \cite{mope}, Theorem 1.30. As explained there, it is sufficient to show that, for any $M> 0$, the event
$$A = \left \{ \exists t \in [0,1), \varepsilon > 0 : \sup_{s \in [t,1], d_\kappa(t,s) < \varepsilon} |Z(s) - Z(t)| \leq M d_\kappa(t,s)^\beta \right\}$$
is a null event. We fix an integer $L > 4$ whose precise value will be specified later. For any $n > 3L$ let
$$R_n = \{ 0 \leq k \leq 2^n - 3L : d_\kappa( (k+3L)2^{-n}, k 2^{-n}) \leq \kappa^{n-L} \}.$$
For $t \in [0,1)$ and $n \in \N$, let $k_n(t) = \lfloor 2^n t \rfloor \in \N$
which satisfies $k_n(t)2^{-n} \leq t < (k_n(t) +1)2^{-n}$. Then, with $t=(t_1, t_2, \ldots)$, choose $m \in \N$ with $t_m = 0$ and set $n = n(m) = m + \lceil \log_2 3L \rceil$.
(Note that there are infinitely many $m$ with this property since $t\ne 1$.)
Then, $d_\kappa( (k_n(t)+3L)2^{-n}, k_n(t) 2^{-n}) \leq \kappa^{n - \lceil \log_2 3L \rceil}$. Hence, as $t \neq 1$, we have $k_n(t) \in R_n$  for infinitely many $n$.
Moreover,  for $k \in R_n$, we also have
$d_\kappa((k+x)2^{-n}, k 2^{-n}) \leq \kappa^{n-L}$ for $0 \leq x \leq 3L$ by monotonicity. Next, let
$$S_{n,k} =\{ |Z((k+3i)2^{-n}) - Z((k+3i-1)2^{-n})| \leq 2 M  \kappa^{(n-L)\beta} \: \forall \: 1 \leq i \leq L\}, \quad S_n = \bigcup_{k \in R_n}S_{n,k}. $$
Assume that $\omega \in A$ and $(t_0, \varepsilon_0) = (t_0(\omega), \varepsilon_0(\omega))$ satisfies the statement in the event $A$. 
Then, if $k_n(t_0) \in R_n$ and $n > 3L$ is large enough such that $\kappa^{n-L} < \varepsilon$, we infer
\begin{align*}
& |Z((k_n(t_0)+3i)2^{-n}) - Z((k_n(t_0)+3i-1)2^{-n})| \\ & \leq |Z((k_n(t_0)+3i)2^{-n}) - Z(t_0)| + |Z((k_n(t_0)+3i-1)2^{-n}) - Z(t_0)| \\
& \leq M (d_\kappa(t_0,(k_n(t_0)+3i)2^{-n})^\beta + d_\kappa(t_0,(k_n(t_0)+3i-1)2^{-n})^\beta) \\
&  \leq M (d_\kappa(k_n(t_0)2^{-n},(k_n(t_0)+3i)2^{-n})^\beta + d_\kappa(k_n(t_0)2^{-n},(k_n(t_0)+3i-1)2^{-n})^\beta) \leq 2M \kappa^{(n-L)\beta}
\end{align*}
for all $1 \leq i \leq L$.  Hence, $\omega \in S_{n,k_n(t_0)}$. As $k_n(t_0) \in R_n$ for infinitely many $n$, we can deduce that also $\omega \in S_n$ for infinitely many $n$, that is $A \subseteq \liminf S_n$. We finish the proof by showing that $\Prob{\liminf S_n} = 0$. For $k \in R_n$,  we have
\begin{align*}
& \Prob{S_{n,k}} = \Prob{|Z((k+3i)2^{-n}) - Z((k+3i-1)2^{-n})|  \leq 2 M  \kappa^{(n-L)\beta}  \: \forall \:1 \leq i \leq L} \\ & = \prod_{i=1}^L \Prob{\gamma^{1/2} (d_\kappa((k+3i)2^{-n}, (k+3i-1)2^{-n}))^{1/2} | N| \leq
 2 M \kappa^{(n-L)\beta}} \\
 & \leq \left( \Prob{\gamma^{1/2} \kappa^{(n-L)/2} |N|\leq 2 M \kappa^{(n-L)\beta}}\right)^{L}
\end{align*}
As the density of $|N|$ is bounded by 2, we have
$$\Prob{S_{n,k}} \leq (4M \kappa^{-L(\beta-1/2)} \gamma^{-1/2})^{L} \kappa^{nL(\beta-1/2)}.$$
Hence, as $|R_n| \leq 2^n$, by an application of the union bound, we see that the sequence $\Prob{S_n}$ is summable upon choosing $L>\max (4, 2/((2-\alpha)(2\beta-1)))$. Thus, $\Prob{\liminf S_n} = 0$ as desired.
\end{proof}

\section{Appendix}

\subsection{Refined information on the mean $\E{X^{(2)}_n(\ell)}$ and $\E{X^{(3)}_n(\ell)}$}
We denote
\begin{align}\label{rn_def_cn}
c_n^{(2)}(\ell) := \E{X^{(2)}_n(\ell)},\quad
c_n^{(3)}(\ell) := \E{X_n^{(3)}(\ell)}, \quad n\ge 1, 1\le \ell \le n.
\end{align}
The following result is sufficient to handle the difference between $2$- and $3$-version of the algorithm. Again, we assume $c = 1$.
\begin{lem} \label{lem3}
For $c_n^{(2)}$ and $c_n^{(3)}$ defined in (\ref{rn_def_cn}) we have $c_n^{(3)}(\ell) \leq c_n^{(2)}(\ell)$ for all $1\le \ell \le n$ and,
as  $n \to \infty$,
\begin{align*}
 \sup_{1\le \ell\le n} c_n^{(2)}(\ell) = \bo(n),\quad
0\le \sup_{1\le \ell\le n}(c_n^{(2)}(\ell) - c_n^{(3)}(\ell)) = \bo\left(n^\alpha \sqrt{\log n} \right).
\end{align*}
\end{lem}
The proof of the lemma makes use of a tail bound for the distribution of $M_n$ given in Lemma \ref{lem:binbeta}. It relies on standard concentration results for sums of independent
random variables. The following simplified version of Bernstein's inequality, see e.g., Theorem 2.8 in \cite{boluma13} is sufficient: For a sequence of independent random variables $X_1, \ldots, X_n$ with $0 \leq X_i \leq 1$ for all $i =1, \ldots, n$, we have
\begin{align}\Prob{\left|\sum_{i=1}^n  X_i - \E{X_i} \right| \geq t} \leq 2 \exp \left( - \frac{t^2}{2 \sum_{i=1}^n \E{X_i} + 2t/3} \right), \label{bernstein}
\end{align}
for all $t > 0$ and $n \in \N$.
\begin{lem} \label{lem:binbeta}
Let $k  \sim n^\alpha$ be odd. There exists a constant $C>0$ such that for all  $y > 0$ and $n\ge n_0$ we have
$$\Prob{M_n - \frac 1 2  > yn^{-\alpha/2}} \le C\exp\left(-\frac{y^2}{4}\right).$$
\end{lem}
\begin{proof}
Using the connection \eqref{verb.binbeta} with $x = 1/2 + y n^{-1/4}$ and $Y_n = \text{Bin}(k,x)$ we infer
$$\Prob{M_n - \frac 1 2  > yn^{-\alpha/2}} \leq \Prob{Y_n - \E{Y_n} \leq \frac 1 2 - k yn^{-\alpha/2}}.$$ We may assume that
$1 \leq y \leq n^{\alpha/2}$. Using Bernstein's inequality \eqref{bernstein}, we can deduce that for all $n$ sufficiently large,
\begin{align*}
\Prob{M_n - \frac 1 2  > yn^{-1/4}} \leq 2 \exp \left( - \frac{(ky n^{-\alpha/2}- \frac 1 2)^2}{k + 8kyn^{-\alpha/2}/3 -1/3} \right).
 \end{align*}
From here, the result follows easily.
\end{proof}

\begin{proof}[Proof of Lemma \ref{lem3}]
The claim $c_n^{(3)}(\ell) \leq c_n^{(2)}(\ell)$ is clear.
Note that \eqref{rec2ver} implies
\begin{align} & c^{(2)}_n(\ell) =  \sum_{i=\ell }^n \Prob{I_n = i} c^{(2)}_{i}(\ell)  +  \sum_{i= 1}^{\ell-1} \Prob{I_n = i}  c^{(2)}_{n-i}(\ell - i)
+  n-k + \E{T_n}  \label{rec:mean2} \end{align}
for all $n \geq n_0$.
Assuming that $c_i^{(2)}(\ell) \leq C i$ for all $1 \leq \ell \leq i$ and $i \leq n-1$, it  follows that
$$c^{(2)}_n(\ell) \leq  C \E{\max(I_n,n-I_n)} +  n-k + \E{T_n}.$$
Choosing $n$ and $C$ large enough, the right hand side is bounded by $Cn$ as $\E{T_n} = \bo(n^\alpha)$ which proves $\sup_{1\le \ell\le n} c_n^{(2)}(\ell) = \bo(n)$.

Much in the same way as \eqref{rec:mean2} follows from  \eqref{rec2ver}, the following recurrence follows from  \eqref{rec3ver}:
\begin{align*} c^{(3)}_n(\ell) =  \sum_{i=\ell + 1}^n \Prob{I_n = i} c^{(3)}_{i-1}(\ell)  +  \sum_{i= 1}^{\ell-1} \Prob{I_n = i}
 c^{(3)}_{n-i}(\ell - i) + n-k + \E{T_n},
  \end{align*}
for $n \geq n_0$.
We proceed recursively and assume that $|c^{(2)}_i(\ell) - c^{(3)}_i(\ell)|   \leq C  i^\alpha \sqrt{\log i}$ for all $1 \leq \ell \leq i$, $i \leq n-1$. Then, denoting $I^{*}_n = \max(I_n, n-I_n)$,
\begin{align}
|c^{(2)}_n(\ell) - c^{(3)}_n(\ell)| &\leq  \sum_{i = \ell}^{n-1} |\Prob{I_n = i+1}  - \Prob{I_n = i}\!| c_i^{(3)}(\ell) \nonumber\\
 &\quad ~+ \Prob{I_n = n} c_n^{(2)}(\ell) + C \E{(I^*_n)^\alpha \sqrt{\log I^*_n}} \nonumber  \\
& =  \sum_{i = \ell}^{n-1} |\Prob{I_n = i+1}  - \Prob{I_n = i}| c_i^{(3)}(\ell) + C \E{(I^*_n)^\alpha \sqrt{\log I^*_n}} \label{step1}
\end{align}
For now, let us assume (a proof given below) that
\begin{align} \label{splitdiff} \Prob{I_n = i+1}  - \Prob{I_n = i} = \bo\left(n^{\alpha-2} \sqrt{\log n}\right)
\end{align}
uniformly in $1 \leq i \leq n$. By the results obtained so far,  we have
\begin{align*} 
c_n^{(3)}(\ell) \leq  c_n^{(2)}(\ell) =\bo(n)
\end{align*}
uniformly in $1 \leq \ell \leq n$. Using these two bounds, it follows from \eqref{step1} that\begin{align*}
|c^{(2)}_n(\ell) - c^{(3)}_n(\ell)| &\leq  \left(C^* n^\alpha + C \E{(I^*_n)^\alpha} \right)\sqrt{\log n}
\end{align*}
for some universal constant $C^* > 0$. Let $\varepsilon > 0$ be sufficiently small and assume $n$ was chosen large enough such that
$\E{ (I^{*}_m/m)^\alpha} \leq 1- \varepsilon$ for all $m \geq n$. Then $|c^{(2)}_n(\ell) - c^{(3)}_n(\ell)| \leq C n^\alpha \sqrt{\log n}$
follows upon choosing $C \geq C^*/\varepsilon$.\\

It remains to prove \eqref{splitdiff}:
First, observe that we can write
\begin{align}
\Prob{I_n = i+1}  - \Prob{I_n = i} = \Prob{I_n=i} \left(\frac{(k-1) (n-2i)}{(n-i)(2i-(k-1))}\right). \label{diffbesser}
\end{align}
By symmetry, it is enough to consider the case $i \leq \lfloor n/2 \rfloor $. Moreover, again by symmetry, $\Prob{I_n = i}$ is maximal for $i = \lfloor n/2 \rfloor$. An application of Stirling's formula in \eqref{weights} shows that
$$\sup_{1 \leq i \leq n} \Prob{I_n = i} = \Prob{I_n = \lfloor n/2 \rfloor} \sim \sqrt{\frac{2}{\pi}} n^{\alpha/2-1}.$$
In particular, $\sup_{1 \leq i \leq n} \Prob{I_n = i} = \bo(n^{\alpha/2-1})$, which is suggested by the limit law in Lemma \ref{lem:split}. Now,
let $C_1 > 0$ (to be specified later) and $\gamma_n = n/2 - C_1 \sqrt{\log n} n^{1-\alpha/2}$. For $i \geq \gamma_n$
it is easy to see that the second factor on the right hand side of \eqref{diffbesser} is uniformly bounded by a constant multiple of
$\sqrt{\log n} n^{\alpha/2-1}$. Thus,
$$\sup_{\gamma_n \leq i \leq \lfloor n/2 \rfloor } \Prob{I_n = i+1}  - \Prob{I_n = i} = \bo(n^{\alpha-2} \sqrt{\log n}).$$
To treat small values of $i$, consider the event $A_n := \{M_n \geq \frac{1}{2} - C_2 \sqrt{\log n} n^{-\alpha/2}\}$ with $0<C_2<C_1$. We have
\begin{align*}
\Prob{I_n < \gamma_n, A_n} &\leq \Prob{\text{Bin}(n-k,  1/2 - C_2 \sqrt{\log n} n^{-\alpha/2}) < \gamma_n - (k+1)/2} \\
& \leq \Prob{|Y_n - \E{Y_n}| \geq (C_1-C_2) \sqrt{\log n} n^{1-\alpha/2}}
\end{align*}
where $Y_n = \text{Bin}(n-k, 1/2 - C_2 \sqrt{\log n} n^{-\alpha/2})$ and $n$ is assumed to be sufficiently large.
Using Bernstein's inequality \eqref{bernstein}, the latter display is bounded by a multiple of $\exp(- (C_1-C_2)^2 n^{1-\alpha} \log n /2)$.
By Lemma \ref{lem:binbeta}, $\Prob{A_n^c} =  \bo(n^{-C_2^2/4})$ which finally shows that
$$\sup_{1 \leq i \leq \gamma_n } \Prob{I_n = i}  = \bo\left(\max\left\{n^{-C_2^2/4},\exp\left(-\frac{1}{2} (C_1-C_2)^2 n^{1-\alpha} \log n \right)\right\}\right).$$
This finishes the proof of the lemma by choosing, e.g., $C_2 = 3$ and $C_1=4$.  \end{proof}

\subsection{Proof of Lemma \ref{lem:helppq}}
By \eqref{vwz} we have  $\sum_{t \in N_f} |\Delta f(t)|^p =: M < \infty$. Let $\delta, \varepsilon > 0$. There exists a number $N \in \N$ and a
set $N_f' = \{\sigma_i: 1 \leq i \leq N\} \subseteq N_f$ with $0 < \sigma_1 < \ldots < \sigma_N \leq 1$
such that, first,
\begin{align*}
 |f(t)-f(s)| \leq \delta \quad & \text{for all} \:  s,t \in [0,1], 1 \leq i \leq N \:\text{with} \: s,t \in [\sigma_i, \sigma_{i+1}) \\
& \text{and for all} \: s,t \in [0, \sigma_1) \: \text{or} \: s,t \in [\sigma_N, 1],
\end{align*}
and second, $\sum_{t \in N_f \backslash N_f'}
|\Delta f(t)|^q < \varepsilon$. For the remainder of the proof  we consider  $t=1$.
Let $\pi \in \Pi(1)$ with $\text{mesh}(\pi) < \min_{0 \leq i \leq N-1} |\sigma_{i+1} - \sigma_i|$ (where $\sigma_0 := 0$) and, if $\sigma_N < 1$, additionally
$\text{mesh}(\pi) < 1-\sigma_N$. For $\tau = \tau_i \in \pi$ with $1 \leq i \leq |\pi|-1$ let $\tau^* = \tau_{i+1}$ be its successor.
For $1 \leq i \leq N-1$, let $\tau_i^- \in \pi$ be the largest element strictly smaller than $\sigma_i$ and $\tau_i^+ = (\tau_i^-)^*$. Then, we have
\begin{align}
& \left|\sum_{i = 0}^{|\pi|-1} \left|f\left(\tau_{i+1}\right) - f\left(\tau_{i}\right)\right|^q - [f]^{(q)}_1 \right| \nonumber \\
& \leq  \left| \sum_{i=1}^{N} \left|f\left(\tau_i^{+}\right) - f\left(\tau_i^{-}\right)|^q - |\Delta f\left(\sigma_i\right)\right|^q \right|
 \label{firstsum}\\
&\quad~+ \sum_{s \in N_f \backslash N_f'} |\Delta f(s)|^q  + \sum_{i=1}^{N} \sum_{\tau \in \pi \cap [\sigma_i, \sigma_{i+1})
\atop \text{with}\  \tau^* \in (\sigma_i, \sigma_{i+1}]} |f(\tau^*) - f(\tau)|^q \label{secsum}
\end{align}
By definition, the first summand in \eqref{secsum} does not exceed $\varepsilon$. Moreover,
\begin{align*}
   \sum_{i=1}^{N} \sum_{\tau \in \pi \cap [\sigma_i, \sigma_{i+1})
\atop \text{with}  \tau^* \in (\sigma_i, \sigma_{i+1}]} |f(\tau^*) - f(\tau)|^q
  \leq \delta^{q-p} \sum_{i=1}^{N} \sum_{\tau \in \pi \cap [\sigma_i, \sigma_{i+1})
\atop \text{with}  \tau^* \in (\sigma_i, \sigma_{i+1}]} |f(\tau^*) - f(\tau)|^p
 \leq \delta^{q-p} V_p(f).
\end{align*}
To treat the term in \eqref{firstsum}, note
that, for all $x , y \in \R$, we have the elementary inequality $$||x+y|^q - |x|^q| \leq |y|^q + 2^q(|x|^{q-1}|y| + |x| |y|^{q-1}).$$
Applying this inequality to the $i$-th summand of \eqref{firstsum} where $x = \Delta  f(\sigma_i)$ and $y = f(\tau_i^{+}) - f(\sigma_i) + f(\sigma_i-) - f(\tau_i^{-})$, the $i$-th summand is bounded from above by
\begin{align*}
\lefteqn{ |\delta_i^+ + \delta_i^-|^q + 2^q(|\Delta  f(\sigma_i)|^{q-1} |\delta_i^+ + \delta_i^-| + |\Delta  f(\sigma_i)| |\delta_i^+ + \delta_i^-|^{q-1})
} \\ & \leq 2^q (|\delta_i^+|^q + |\delta_i^-|^q) + 2^q(|\Delta  f(\sigma_i)|^{q-1} (|\delta_i^+| + |\delta_i^-|))\\
 &\quad~+ 4^{q} |\Delta  f(\sigma_i)| (|\delta_i^+|^{q-1} + |\delta_i^-|^{q-1}),
\end{align*}
where we have set $\delta_{i}^+ = f(\tau_i^{+}) - f(\sigma_i)$ and $\delta_{i}^- =  f(\sigma_i-)- f(\tau_i^{-}).$ It is now
straightforward to show that the sum (over $1 \leq i \leq N$) of the last display is bounded from above by $C \delta^\alpha$ for some $0 < \alpha  = \alpha(p,q) < 1$ and
$C = C(p,q) > 0$.
This finishes the proof of \eqref{assertionlemma} as $\delta$ and $\varepsilon$ were chosen arbitrarily. Exemplarily, we pick one of the terms. We have
\begin{align*}
\sum_{i=1}^{N} |\Delta  f(\sigma_i)|^{q-1} |\delta_i^+| & \leq \left(\sum_{i=1}^{N} |\Delta  f(\sigma_i)|^{q}\right)^{1-1/q}
\left(\sum_{i=1}^{N} |\delta_i^+|^q \right)^{1/q} \\
& \leq \delta^{1-p/q} \max_{t \in N_f} |\Delta f(t)|^{(q-p)(1-1/q)} M^{1-1/q} V_p^{1/q}(f).
\end{align*}
The regularity of $t \mapsto [f]_t^{(q)}$ and the characterization of its jumps follow immediately.

\end{document}